\documentclass[11pt]{article}
\usepackage[toc,title,page]{appendix}
\usepackage[usenames]{color}
\usepackage{amssymb,amsmath,latexsym}
\numberwithin{equation}{section}
\usepackage{amsthm}
\usepackage[active]{srcltx}
\usepackage[mathscr]{eucal}
\usepackage{setspace}
\usepackage[color=yellow,textsize=small]{todonotes}
\usepackage{geometry}
\usepackage{multirow}
\usepackage{graphicx}
\usepackage{float}
\usepackage{amsfonts}
\usepackage{amssymb}
\usepackage{amsthm}
\usepackage{newlfont}
\usepackage{stmaryrd}
\usepackage{mathrsfs}
\usepackage[colorlinks,
            linkcolor=blue,
            anchorcolor=blue,
            citecolor=blue
            ]{hyperref}
\usepackage{indentfirst}
\usepackage{cases}
\usepackage{bookmark}
\usepackage{graphicx}
\usepackage{pdfpages}
\usepackage{setspace}
\usepackage{abstract}
\usepackage{cite}

\newsavebox{\mybox}

\allowdisplaybreaks[4]

\newtheorem{thm}{Theorem}[section]
\newtheorem{cor}[thm]{Corollary}
\newtheorem{lem}[thm]{Lemma}
\newtheorem{prop}[thm]{Proposition}
\newtheorem{defi}[thm]{Definition}

\newtheorem{rem}[thm]{Remark}

\def\theequation{\arabic{section}.\arabic{equation}}

\def\${|\!|\!|}
\def\({\left(}
\def\){\right)}

\newcounter{bean}
\newcommand{\benuma}{\setlength{\labelwidth}{.25in}
	\begin{list}%
		{(\alph{bean})}{\usecounter{bean}}}
	\newcommand{\eenuma}{\end{list}}

\newcommand{\bi}{\begin{itemize}}
	\newcommand{\ei}{\end{itemize}}
\newcommand{\be}{\begin{enumerate}}
	\newcommand{\ee}{\end{enumerate}}
\newcommand{\beqs}{\begin{equation*}}
\newcommand{\eeqs}{\end{equation*}}
\newcommand{\beq}{\begin{equation}}
    \newcommand{\eeq}{\end{equation}}
    \newcommand{\bald}{\begin{aligned}}
        \newcommand{\eald}{\end{aligned}}
\newcommand{\beqys}{\begin{eqnarray*}}
	\newcommand{\eeqys}{\end{eqnarray*}}
\newcommand{\beqy}{\begin{eqnarray}}
\newcommand{\eeqy}{\end{eqnarray}}

\newcommand{\dif}{\mathrm{d}}

\newcommand{\E}{\mathbb{E}}

\mathchardef\mhyphen="2D
\newcommand{\bp}{\begin{pmatrix}}
	\newcommand{\ep}{\end{pmatrix}}
\begin{document}
\title{ \bf Robust pointwise  second-order necessary conditions for singular stochastic optimal control with model uncertainty  }
\date{}

\author{Guangdong Jing\thanks{ 
Guangdong Jing  \\ \indent  School of Mathematics and Statistics,   \\ \indent 
Beijing Institute of Technology \\  \indent Beijing 100081, China \newline  \indent  
{jingguangdong@mail.sdu.edu.cn}} }

\maketitle
\begin{abstract}
  \noindent
  We investigate the singular stochastic optimal control problem with model uncertainty, where the necessary conditions determined by the corresponding maximum principle are trivial. We derive robust integral form and pointwise second-order necessary optimality conditions involving integrals with respect to certain reference probabilities, under specific compactness and monotonicity. Both the drift and diffusion terms depend on the control, and the control regions are assumed to be convex because saddle point analyses are crucial for deriving the variational inequality with a common reference probability for any admissible control. Other main technical components for the integral type conditions include weak convergence arguments and the minimax theorem, while the pointwise conditions involve the Clark-Ocone formula and Lebesgue differentiation type theorem. Additionally, an example is provided to demonstrate the motivation and effectiveness of the results.
\end{abstract}

\bigskip

\noindent{\bf AMS  subject classifications}.   93E20, 60H07, 60H10

\bigskip

\noindent{\bf Key Words}. Stochastic optimal control, Model uncertainty,  Maximum principle, Second-order necessary conditions, Clark-Ocone formula 

\section{Introduction}

In optimal control theory, one of the most important methods is maximum principle. It asserts that the optimal control should maximize the corresponding Hamiltonian $H$, and actually is a first-order necessary condition for the optimality. Formally, 
\[H(t,\overline{x},\overline{u})=\max_{v\in U}H(t,\overline{x},v)\quad  a.e.\  (t,\omega )\in [0,T]\times \Omega. \]
Related works in ordinary differential equations setting can date back to~\cite{PontryaginBoltyanskiiGamkrelidzeMishchenko1962}.
However, there are singular cases in both deterministic and stochastic settings with different Hamiltonians. Taking the stochastic case as an example, 
the optimal contol $\bar u(\cdot)$  is called singular optimal control  in the classical sense, if 
the gradient and the Hessian of the Hamiltonian with respect to the control variable are degenerate, i.e. 
\[H_u(t,\overline{x},\overline{u})=0, \quad H_{uu}(t,\overline{x},\overline{u})=0,\quad a.e.\  (t,\omega )\in [0,T]\times \Omega.\]
In the singular case, the first-order necessary  condition can not provide enough information for the theoretical analysis and numerical computing.
In the face of this situation,  works on the second-order necessary conditions for deterministic optimal control problems follow closely \cite{GabasovKirillova1972,warga1978jde}.
We mention that there are also results about second-order necessary conditions for deterministic optimal control problem on Riemannian manifolds~\cite{cuiqingdenglizhangxu2016crmasp,cuiqingdenglizhangxu2019esaim,denglizhangxu2021jde} and recommend \cite{FrankowskaHoehener2017jde, FrankowskaOsmolovskii2020scl, FrankowskaOsmolovskii2019amo, FrankowskaOsmolovskii2018siam, Hoehener2013amo,louhongwei2010dcds, louhongweiyongjiongmin2018mcrf, Nguyen2019p, Osmolovskii2020jota, Osmolovskii2018jmaa} among numerous others for the second-order  necessary conditions of  deterministic optimal control problem to interested readers.

Recently, the second-order necessary conditions for singular \emph{stochastic} optimal control problems also have attracted rather extensive attention. For example, in the classical stochastic differential equations setting: 
\cite{tangshanjian2010dcds} the diffusion term of the control system is control independent and the control regions are allowed to be nonconvex; 
\cite{zhanghaisenzhangxu2015siam} with convex control constraint and diffusion term contains control;
further \cite{Frankowskazhanghaisenzhangxu2017jde,zhanghaisenzhangxu2017siam,zhanghaisenzhangxu2018siamreview,zhanghaisenzhangxu2016scm} with general control constraint and control dependent diffusion term;
and further \cite{Frankowskazhanghaisenzhangxu2019transams,Frankowskazhanghaisenzhangxu2018siam} with additional state constraints; \cite{BonnansSilva2012amo} with either convex control constraints or finitely many equality and inequality constraints over the final state.
The second-order necessary optimality conditions in  controlled stochastic evolution equations setting: \cite{luqi2016conference,luqizhanghaisenzhangxu2021siam} with convex control constraint and diffusion term contains control; further \cite{Frankowskazhangxu2020spa} with more general control constraint;
and further \cite{Frankowskaluqi2020jde} with adding control and state constraints.
A second-order stochastic maximum principle for generalized mean-field singular control problem  \cite{guohanchengxiongjie2018mcrf} with diffusion term independent of control variable.
A second-order maximum principle for singular optimal controls with recursive utilities of stochastic systems \cite{dongyuchaomengqinxin2019jota} and stochastic delay systems \cite{haotaomengqingxin2019ejc}, in both of which  the control domain is nonconvex and the diffusion coefficients are independent of control.

It is well-known that there are essential difficulties in stochastic optimal control problems compared with the deterministic setting. Typically, when the diffusion term of the state equation involves control with nonconvex control constraint, in general the Ekeland topology is needed, under which the estimation $\mathbb{E} \big|\int_{0}^{T} \chi_{E_\varepsilon} \mathrm{d}W(t)\big|^2=\varepsilon$ is weaker than that $ \big|\int_{0}^{T} \chi_{E_\varepsilon} \mathrm{d}t\big|^2=\varepsilon^2$  in the deterministic setting. It causes problems in the estimation of the  Taylor expansion of the cost functions.
Therefore,  the study on the  first-order necessary conditions of stochastic optimal control problems for the general case was not achieved until~\cite{pengshige1},  motivated by the theory of non-linear backward stochastic differential equations \cite{pardouxpengshige} and the additionally introduced adjoint equations, although for some special cases it can be traced back to \cite{KushnerSchweppe1964,KushnerHJ}.
In particular, results on the second-order necessary conditions for singular stochastic optimal control problems were firstly presented much later in~\cite{MahmudovBashirov1997}.

Note that all the results mentioned above for singular optimal stochastic control problems are not robust.
In order to obtain robust results, in this paper we  take into account the model uncertainty originally proposed by Hu and Wang \cite{humingshangwangfalei2020siam}.
Generally speaking, for the It\^{o} stochastic differential equations (which are  powerful in modelling stochastic dynamics among economics, molecular biology, meteorology, etc.),  the drift and diffusion coefficients  are not exactly known in general. We may guess that it obeys certain family of distributions, but we do not know in advance what a distribution should occur. 
For example, the controlled systems are different for bull markets or bear markets  relatively, and the corresponding cost functions are also different, denoting respectively by $\mathcal J_1(u(\cdot))$ and $\mathcal J_2(u(\cdot))$. But actually the probability for experiencing a bull market is unknown. One should characterize the cost functional by \[\sup_{ \mathfrak{i}  \in[0,1]}\left\{ \mathfrak{i}  \mathcal J_1(u(\cdot))+(1- \mathfrak{i}  )\mathcal J_2(u(\cdot))\right\}.\]
See also He, Luo, and Wang \cite{hlw23} for a mean-variance portfolio selection model under such a model uncertainty setting. 
We formally introduce the formulation here, and it will be made rigorous afterwards.  $\Gamma$ is a locally compact Polish space and $\gamma\in\Gamma$ represents  different market conditions. The controlled state process is governed by
\begin{equation}\label{state-equa-gene-origi}
\left\{
\begin{aligned}
&\mathrm{d}x_\gamma(t)=b_\gamma(t,x_\gamma(t),u(t))\mathrm{d}t
    +\sigma_\gamma(t,x_\gamma(t),u(t))\mathrm{d}W(t),    \quad   t\in[0,T],       \\
&x_\gamma(0)=x_0.
\end{aligned}
\right.
\end{equation}
Let $\Lambda=\{\lambda \}$ be the set of all possible probability distributions of the uncertainty parameter $\gamma$. For any given  parameter $\gamma$ and measure $\lambda $, the cost functionals are formulated as 
\begin{equation*}
   \mathcal J(u(\cdot);\gamma ):= \mathbb{E}\Big(\int_0^Tf_\gamma(t,x_\gamma(t),u(t))\mathrm{d}t
         +h_\gamma(x_\gamma(T))\Big) ,
  \end{equation*} 
  and then denote 
  \begin{equation*}
    J(u(\cdot);\lambda ):= \int_{\Gamma} \mathcal  J(u(\cdot);\gamma )\lambda(\mathrm{d}\gamma).
  \end{equation*}
Taking the above model uncertainty into consideration, the robust cost functional has the following form
\begin{equation}\label{cost-function-origi}
  J(u(\cdot)):=\sup_{\lambda\in\Lambda} J(u(\cdot);\lambda ).
\end{equation}
It involves a supremum over a family of probability measures. 
Note that the measurability of $\mathcal  J(u(\cdot);\gamma ), \gamma \in \Gamma $, and the well-posedness of $\sup_{\lambda\in\Lambda} J(u(\cdot);\lambda )$ are all essential problems that will be proved later. Now we introduce the concept of admissible control before presenting the optimal control problems. 

\begin{defi}
An $\mathbb{F}$-progressively measurable stochastic process $u(\cdot):[0,T]\times\Omega\to U\subset\mathbb{R}^m$ is called an admissible control, if $u(\cdot)\in \mathcal{U}^\beta:=\{u(\cdot)\in L_\mathbb{F}^\beta (\Omega; L^\beta ([0,T]; \mathbb{R}^m)), \text{ valued in}\  U\}$ for $\beta \ge1$. 
\end{defi}
The singular stochastic optimal control problem is to find an admissible $\bar{u}(\cdot)\in\mathcal{U}^\beta $, such that
\begin{equation}\label{jujubint}
J(\bar{u}(\cdot))=\inf_{u(\cdot)\in\mathcal{U}^\beta }J(u(\cdot)).
\end{equation}
Any $\bar{u}(\cdot)\in\mathcal{U}^\beta $ satisfying \eqref{jujubint} is called an optimal control with the initial state $x_0$. The corresponding state ${\bar x}_\gamma(\cdot;\bar u(\cdot))$ is called an optimal state, and $(\bar{u}(\cdot),{\bar x}_\gamma(\cdot))$ an optimal pair. 
Formally, by taking Taylor expansions of the variational cost functional $J(u^\varepsilon(\cdot))$ $=$ $J(\bar{u}(\cdot)+\varepsilon (u(\cdot)-\bar{u}(\cdot)))$ with respect to the perturbation parameter $\varepsilon $ in the convex control region, we obtain 
\begin{equation*}
  J(u^\varepsilon )=J(\bar{u} )+\varepsilon\,  \mathbb E\int_{0}^{T} \psi_{1}(u ,\bar{u} ) \mathrm d t+\varepsilon ^2 \,  \mathbb E\int_{0}^{T}  \psi_{2}(u ,\bar{u} )\mathrm d t +o (\varepsilon^2) \quad \text{as} \  \varepsilon \to0.
\end{equation*}
Based on the above expression, a singular optimal control problem in the classical sense refers to the case where the optimal control cannot be uniquely determined by the first-order necessary optimality condition $\mathbb E\int_{0}^{T} \psi_{1}(u ,\bar{u} ) \mathrm d t \ge 0$. In contrast, $\mathbb E\int_{0}^{T} \psi_{1}(u ,\bar{u} ) \mathrm d t=0$ holds trivially. 
Therefore, we should naturally further study the second-order necessary conditions by analyzing in details the limit
\begin{align*}
    \lim_{\varepsilon \to 0}\frac{1}{\varepsilon^2} \(J(u^\varepsilon )-J(\bar{u} )\)
\end{align*}
characterized by the integral type variational inequality 
\begin{equation*}
    \begin{aligned}
        \mathbb E\int_{0}^{T}  \psi_{2}(u ,\bar{u} )\mathrm d t\ge0, \quad  \forall u\in \mathcal{U}^\beta,
    \end{aligned}
\end{equation*}
or further by the pointwise condition  
\begin{equation*}
      \begin{aligned}
        \psi_{2}(v,\bar{u} )\ge0, \quad  \forall v\in U \  a.s.\ a.e.  
      \end{aligned}
\end{equation*}

Inspired by \cite{hlw23,humingshangwangfalei2020siam,luqizhanghaisenzhangxu2021siam,zhanghaisenzhangxu2015siam,zhanghaisenzhangxu2018siamreview}, we formulate the \textit{singular} (see Definition \ref{singularoptimalcontroldef}) stochastic optimal control problems \eqref{state-equa-gene-origi}-\eqref{jujubint} in the context of model uncertainty.

The objective of this paper is to establish the well-posedness of the formulation, and to derive its second-order necessary optimality conditions. 
We will first derive the conditions in integral form and then establish pointwise conditions with additional assumptions, as the latter are more straightforward to verify in practical applications.
Intuitively, characterizing the second-order necessary conditions involves studying the exact expressions of the $\psi_{1}(u(\cdot),\bar{u}(\cdot))$ and $\psi_{2}(u(\cdot),\bar{u}(\cdot))$. To achieve the purpose, we assume that the control region $U$ is convex and the set of probability measures $\Lambda$ on $(\Gamma,\mathcal{B}(\Gamma))$ is weakly compact and convex. 
Apart from standard linearization techniques, duality principles, and density arguments, we employ methods as detailed in the aforementioned inspiring literature such as weak convergence analysis, saddle point analysis, and Malliavin calculus. Based on these analyses and by developing targeted approaches, we derive a variational inequality that involves the integral 
of a nonlinear combination of variational processes, adjoint processes, derivatives of Hamiltonian, and other coefficients 
with respect to certain reference probabilities. 
The results are robust, enhancing their practical applicability. We provide an example in Section~\ref{secexam} to illustrate the underlying motivation and demonstrate the effectiveness of the results.

To conclude the introduction, we would like to make some remarks on our results and strategies that differ from the existing literature. 
The formulations \eqref{state-equa-gene-origi}-\eqref{jujubint} essentially represent a delicate ``$\inf\sup$ problem'' due to the supremum in the cost functionals taken over a family of probability measures. Therefore, we need to apply weak convergence analysis as that in \cite{hlw23,humingshangwangfalei2020siam} alongside classical variational analysis. 
Additionally, in the singular formulation, we need to further establish the measurability and regularity of the second-order derivatives of the value function and related items arising from it.
The presence of these second-order derivatives and the second variational processes also necessitates a nonlinear framework since the crucial linear relations found in the first-order linearized process do not hold.
To derive the pointwise variational inequality with a common reference probability for any admissible control, we conduct saddle point analyses and weak convergence arguments with certain monotonicity conditions regarding the controls. 
Additionally, due to the uncertainty inherent in the control model, we must establish uniform regularities involving the uncertainty parameter $\gamma$ to manage the remainder terms $o (\varepsilon^2)$ effectively throughout the variational analysis. 
What's more, to derive the pointwise necessary optimality conditions from the integral type by performing density arguments thoroughly, including temporal, spatial, and sample variables, we employ tools in Malliavin calculus first proposed in \cite{zhanghaisenzhangxu2015siam} and extended in \cite{zhanghaisenzhangxu2018siamreview,luqizhanghaisenzhangxu2021siam}, to cope with obstacles related to the integrable order deficiencies for multiple integrals of mixed Lebesgue and It\^o types for terms emerging in $\psi_{2} (u(\cdot), \bar{u} (\cdot))$ like 
$\int_\Gamma \mathbb{E} \int_{\tau_1}^{\tau_2} {\langle \phi_1(t;\gamma )\int_{\tau_1}^t\phi_2(s;\gamma )\mathrm{d}W(s),\phi_3(t;\gamma ) \rangle} \mathrm{d}t  \lambda(\mathrm{d}\gamma)$. To apply the methods, we establish lemmas of Lebesgue differentiation and Malliavin approximation tailored to our model, both of which are crucial in the density arguments.

The rest of this paper is organized as follows.
Section~\ref{semalli} is devoted to introducing preliminaries in Malliavin calculus.
Section~\ref{seformmainres} rigorously formulates the singular stochastic optimal control problem and presents the principal results alongside some essential preliminary results. 
An example is provided in Section~\ref{secexam} to illustrate the motivation and effectiveness of the results. The proof of the main results are completed in Section~\ref{secprothfirinter} and Section~\ref{seclapoisecpro} separately. Several fundamental lemmas are provided in appendix.

\subsection*{Notations}
Let $(\Omega,\mathcal{F},\mathbb{F},\mathbb{P})$ be a complete filtered probability space, on which a standard one-dimensional Brownian motion $W={\{W_t\}}_{t\in[0,T]}$ is defined, and $\mathbb{F}={\{\mathcal{F}_t\}}_{t\in[0,T]}$ is the natural filtration of $W$ augmented by all the $\mathbb{P}$-null sets in $\mathcal{F}$ where $T>0$ is any fixed time horizon. Denote by $\mathbb{E}[\cdot]$ the expectation with respect to $\mathbb{P}$.
Denote by $\mathbb{R}^{m\times n}$ the space of all $m\times n$ real matrices. For any $A\in\mathbb{R}^{m\times n}$, denote by $A^\top$  the transpose of $A$ and  by $|A|$ the norm of $A$ where $|A|:=\sqrt{tr\{AA^\top\}}$. Denote by $\langle\cdot,\cdot\rangle$ the inner product in $\mathbb{R}^n$ or $\mathbb{R}^m$ without confusion. Denote by $\mathbf{S}^n$ the set of $n\times n$ symmetric real matrices.
Denote by $\partial_{(x,u)^2}\varphi (t,x,u)$ the Hessian of $\varphi $ with respect to $(x,u)$ at $(t,x,u)$. 
$|\xi_1| \lesssim \xi_2$ means that there is a constant $\mathcal{C}\ge0$, such that $|\xi_1| \le \mathcal{C} \xi_2$. $|\xi|=O(1)$ means that there are non-negative constants $\mathcal{C}_1, \mathcal{C}_2$, such that $\mathcal{C}_1 \le |\xi| \le \mathcal{C}_2$. $|\xi|=O(\varepsilon ),\ \varepsilon \to 0^+ $ means that there are non-negative constants $\mathcal{C}_1, \mathcal{C}_2$, such that $\mathcal{C}_1 \varepsilon  \le |\xi| \le \mathcal{C}_2 \varepsilon$, $ \varepsilon \to 0^+ $. That for $|\xi|=O(\varepsilon^2 )$ and $|\xi|=O(\varepsilon^3 ),\ \varepsilon \to 0^+$ are similar.

For $\varphi(\cdot):\mathbb{R}^n\to\mathbb{R}$, $\partial_x\varphi(\cdot)=(\frac{\partial\varphi}{\partial x_1}(\cdot),\ldots,\frac{\partial\varphi}{\partial x_n}(\cdot))^\top$ is a column vector, and for $\varphi(\cdot):\mathbb{R}^m\times\mathbb{R}^n\to\mathbb{R}$, $x\in\mathbb{R}^n$, $u\in\mathbb{R}^m$, denote $\partial_{xu}\varphi(\cdot)$ the multiple partial derivatives of $\varphi$ with respect to $(x,u)$, a $m\times n$ matrix valued map defined on $\mathbb{R}^m\times \mathbb{R}^n$.
For any $\alpha,\beta\in[1,\infty)$, $t\in[0,T]$, denote by

\begin{itemize}
  \item $L_{\mathcal{F}_t}^\beta(\Omega;\mathbb{R}^n)$ the space of $\mathbb{R}^n$-valued $\mathcal{F}_t$-measurable random variable $\xi$ satisfying
$\|\xi\|_\beta:=\(\mathbb{E}|\xi|^\beta\)^{\frac{1}{\beta} }<\infty$;
\item  $L_\mathbb{F}^\beta (\Omega; L^\alpha ([0,T]; \mathbb{R}^n))$  the space of $\mathbb{R}^n$-valued $\mathcal{B}([0,T])\otimes\mathcal{F}$-measurable, and $\mathbb{F}$-adapted  processes $\varphi$ on $[0,T]$, satisfying $\|\varphi\|_{\alpha ,\beta }:=\big[\E(\int_{0}^{T}|\varphi(t)|^\alpha \dif t)^{\frac{\beta }{\alpha } }\big]^{\frac{1}{\beta } }<\infty$;
\item $L_\mathbb{F}^\beta (\Omega; C([0,T]; \mathbb{R}^n))$  the space of $\mathbb{R}^n$-valued $\mathcal{B}([0,T])\otimes\mathcal{F}$-measurable, and $\mathbb{F}$-adapted continuous processes $\varphi$ on $[0,T]$, satisfying
$\|\varphi\|_{\infty,\beta}:=\big[\mathbb{E}\big(\sup_{t\in[0,T]}\left|\varphi(t)\right|^\beta\big)\big]^\frac{1}{\beta}<\infty;$
\item $L^\infty ([0,T]\times\Omega;\mathbb{R}^n)$ the space of $\mathbb{R}^n$-valued $\mathcal{B}([0,T])\otimes\mathcal{F}$-measurable processes $\varphi$ satisfying
$\|\varphi\|_{\infty}:=\mathop{\text{ess}\sup}_{(t,\omega)\in[0,T]\times\Omega}\left|\varphi(t,\omega)\right|<\infty.$
\end{itemize}
Besides,
$L_\mathbb{F}^\beta (\Omega; C([0,T]; \mathbf{S}^n))$, $L_\mathbb{F}^\beta (\Omega; L^2(0,T; \mathbf{S}^n))$ are similarly defined.

\section{Preliminaries in Malliavin calculus}\label{semalli}
A detailed introduction to Malliavin calculus can be found in the textbook~\cite{NualartDavidEulalia}. Here only some elementary definitions and symbols to be used are presented.

For any $h\in L^2(0,T)$, denote $W(h)=\int_0^T h(t)\mathrm{d}W(t)$.
For any smooth $\mathbb{R}^n$-valued random variable of the form  
\[F=\sum_{j=1}^{n}f_j(W(h_1),W(h_2),\ldots,W(h_{j_m})) , \]
where 
$f_j\in C_b^\infty(\mathbb{R}^{j_m}), h_{j_k}\in L^2(0,T), j=1,\dots,n, j_k=1,2,\ldots,j_m, n,j_m\in\mathbb{N}$, the Malliavin derivative of $F$ is defined as 
\[\mathcal{D}_t F=\sum_{j=1}^n\sum_{j_k=1}^{j_m}h_{j_k}(t)\frac{\partial f_j}{\partial x_{j_k}}\left(W(h_1),W(h_2),\ldots,W(h_{j_m})\right) .\]
Then $\mathcal{D} F$ is a smooth random variable with values in $L^2(0,T; \mathbb{R}^n)$.
Denote by $ \mathbb{D}^{1,2}(\mathbb{R}^n)$  the completion of the set of the above smooth random variables with respect to the norm 
\[|F|_{\mathbb{D}^{1,2}}=\Big(\mathbb{E}|F|^2+\mathbb{E} \int_0^T|\mathcal{D}_t F|^2 {d}t \Big)^{\frac{1}{2}}.\]

For $F\in\mathbb{D}^{1,2}(\mathbb{R}^n)$, the following Clark-Ocone representation formula holds:
\begin{equation}\label{claocoori}
  F=\mathbb{E} F+\int_0^T\mathbb{E}\left(\mathcal{D}_t F | \mathcal{F}_t\right)\mathrm{d}W(t).
\end{equation}

Define $\mathbb{L}^{1,2}(\mathbb{R}^n)$ to be the space of processes $\varphi\in L^2([0,T]\times \Omega;\mathbb{R}^n)$, such that

(i) for a.e. $t\in[0,T]$, $\varphi(t,\cdot)\in\mathbb{D}^{1,2}(\mathbb{R}^n)$,

(ii) the function $\mathcal{D}_\cdot\varphi(\cdot,\cdot):[0,T]\times[0,T]\times\Omega\to\mathbb{R}^n$ admits a $\mathcal{B}([0,T]\times[0,T])\otimes\mathcal{F}_T$-measurable version,

(iii) $\$\varphi\$_{1,2}:=\left(\mathbb{E}\int_0^T|\varphi(t)|^2\mathrm{d}t +\mathbb{E}\int_0^T\int_0^T|\mathcal{D}_s\varphi(t)|^2\mathrm{d}s\mathrm{d}t\right)^{\frac{1}{2}}<\infty$.

The set of all $\mathbb{F}$-adapted processes in $\mathbb{L}^{1,2}(\mathbb{R}^n)$ is denoted by $\mathbb{L}_{\mathbb{F}}^{1,2}(\mathbb{R}^n)$.

Moreover, define
\begin{equation*}
\begin{aligned}
\mathbb{L}_{2^+}^{1,2}(\mathbb{R}^n):=\Big\{& \varphi(\cdot)\in\mathbb{L}^{1,2}(\mathbb{R}^n) \mid \exists\ \mathcal{D}^+\varphi(\cdot)\in L^2([0,T]\times \Omega;\mathbb{R}^n) \ \text{such that}  
\\&   
g_\alpha (s):=\sup_{s<t<(s+\alpha)\wedge T}\mathbb{E}|\mathcal{D}_s\varphi(t)-\mathcal{D}^+\varphi(s)|^2<\infty\quad  a.e.\  s\in[0,T],   \\
&  g_\alpha(\cdot)\ \text{is measurable on}\ [0,T]\ \text{for any}\ \alpha>0,\ \text{and}\ \lim_{\alpha\to0^+}\int_0^T g_\alpha(s)\mathrm{d}s=0 \Big\},
\end{aligned}
\end{equation*}
and
\begin{equation*}
\begin{aligned}
\mathbb{L}_{2^-}^{1,2}(\mathbb{R}^n):=\Big\{& \varphi(\cdot)\in\mathbb{L}^{1,2}(\mathbb{R}^n) \mid \exists\ \mathcal{D}^-\varphi(\cdot)\in L^2([0,T]\times \Omega;\mathbb{R}^n) \ \text{such that}  
\\&   
g_\alpha(s):=\sup_{(s-\alpha)\vee0<t<s}\mathbb{E}|\mathcal{D}_s\varphi(t)-\mathcal{D}^-\varphi(s)|^2<\infty\quad  a.e.\  s\in[0,T],   
\\&  
g_\alpha(\cdot)\ \text{is measurable on}\ [0,T]\ \text{for any}\ \alpha>0,\ \text{and}\ \lim_{\alpha\to0^+}\int_0^T g_\alpha(s)\mathrm{d}s=0 \Big\}.
\end{aligned}
\end{equation*}
Let $\mathbb{L}_{2}^{1,2}(\mathbb{R}^n)=\mathbb{L}_{2^+}^{1,2}(\mathbb{R}^n)\cap\mathbb{L}_{2^-}^{1,2}(\mathbb{R}^n)$.
For any $\varphi\in\mathbb{L}_{2}^{1,2}(\mathbb{R}^n)$, denote $\nabla\varphi(\cdot)=\mathcal{D}^+\varphi(\cdot)+\mathcal{D}^-\varphi(\cdot)$.
$\mathbb{L}_{2,\mathbb{F}}^{1,2}(\mathbb{R}^n)$ consists of all the adapted processes in $\mathbb{L}_{2}^{1,2}(\mathbb{R}^n)$.

\section{Formulation and main results}\label{seformmainres}
Consider the following controlled stochatic system:
\begin{equation}\label{state-equa-gene-main}
\left\{
\begin{aligned}
&\mathrm{d}x_\gamma(t)=b_\gamma(t,x_\gamma(t),u(t))\mathrm{d}t
    +\sigma_\gamma(t,x_\gamma(t),u(t))\mathrm{d}W(t),    \quad   t\in[0,T],\gamma\in\Gamma,       \\
&x_\gamma(0)=x_0,
\end{aligned}
\right.
\end{equation}
and the robust cost functional:
\begin{equation}\label{cost-function-main}
  J(u(\cdot)):=\sup_{\lambda\in\Lambda}\int_{\Gamma}\mathbb{E}\Big[\int_0^T f_\gamma(t,x_\gamma(t),u(t))\mathrm{d}t
       +h_\gamma(x_\gamma(T))\Big]\lambda(\mathrm{d}\gamma).
\end{equation}
The singular stochastic optimal control problem is to minimize the robust cost functional $J(u(\cdot))$ over $u(\cdot)\in\mathcal{U}^\beta $:
\begin{equation}\label{cpmc}
    J(\bar{u}(\cdot))=\inf_{u(\cdot)\in\mathcal{U}^\beta }J(u(\cdot)).
\end{equation}
The term \emph{singular} is interpreted strictly in Definition~\ref{singularoptimalcontroldef}, which in brief characterizes the case where the gradient and the Hessian of the Hamiltonian with respect to  the control variable are degenerate.
In this case the first-order necessary conditions can not provide enough information for the theoretical analysis and numerical computation to uniquely determine the optimal control, then it is worth studying the corresponding second-order necessary conditions.
In this paper, results are obtained under  the following assumptions:
\begin{description}
\item[(H1)]  The control region $U(\subset\mathbb{R}^m, m\in\mathbb{N})$ is nonempty and convex. 
\item[(H2)] $\Gamma$ is a locally compact Polish space with distance $ \mathsf{d} $, where $\gamma\in\Gamma$ is the market uncertainty parameter. $\Lambda$ is a weakly compact and convex set of probability measures on $(\Gamma,\mathcal{B}(\Gamma))$.
  \item[(H3)] (i) For any $ \gamma\in\Gamma$,  $b_\gamma:[0,T]\times \mathbb{R}^n\times U\to \mathbb{R}^n$,
$\sigma_\gamma:[0,T]\times \mathbb{R}^n\times U\to \mathbb{R}^n$,
$f_\gamma:[0,T]\times \mathbb{R}^n\times U\to \mathbb{R}$,
$h_\gamma: \mathbb{R}^n \to \mathbb{R}$,
are Borel measurable functions.

(ii) For any $\gamma\in\Gamma$ and $t\in[0,T]$, the functions $b_\gamma, \sigma_\gamma, f_\gamma, h_\gamma$
are continuously differentiable with respect to $(x,u)$ up to the second order.   Besides, $b, \sigma$ and their partial derivatives are uniformly bounded.

(iii) 
There exists $L>0$  such that for any $t\in[0,T], x,x'\in\mathbb{R}^n, u,u'\in U, \gamma\in\Gamma$,
\begin{equation*}
  \left\{
  \begin{aligned}
  &|b_\gamma(t,x,u)-b_\gamma(t,x',u')|+\left|\sigma_\gamma(t,x,u)-\sigma_\gamma(t,x',u')\right| \le  L|x-x'|,  \\ 
  & |\partial_{(x,u)^2}b_\gamma(t,x,u)-\partial_{(x,u)^2}b_\gamma(t,x',u')|\le L(|x-x'|+|u-u'|), \\ 
  & |\partial_{(x,u)^2}\sigma_\gamma(t,x,u)-\partial_{(x,u)^2}\sigma_\gamma(t,x',u')|\le L(|x-x'|+|u-u'|).
  \end{aligned}
  \right.
\end{equation*}

(iv)  There exists $L>0$  such that for any $t\in[0,T], x,x'\in\mathbb{R}^n, u,u'\in U, \gamma\in\Gamma$,
\begin{equation*}
    \left\{
    \begin{aligned}
&|f_\gamma(t,x,u)|\le L(1+|x|^2+|u|^2), \quad  |h_\gamma(x)|\le L(1+|x|^2 ),\\
&|\partial_x f_\gamma(t,x,u)|+|\partial_u f_\gamma(t,x,u)|\le L(1+|x|+|u|), \\
&|\partial_{xx} f_\gamma(t,x,u)|+|\partial_{xu} f_\gamma(t,x,u)|+|\partial_{uu} f_\gamma(t,x,u)|\le L, \\
& 
|\partial_{(x,u)^2}f_\gamma(t,x,u)-\partial_{(x,u)^2}f_\gamma(t,x',u')|\le L(|x-x'|+|u-u'|), 
\\ 
&|\partial_x h_\gamma(x)| \le L(1+|x| ), \quad |\partial_{xx} h_\gamma(x)| \le L.
    \end{aligned}
    \right.
  \end{equation*}

\item[(H4)] For each $N>0$, there exists a modulus of continuity $\overline{\omega}_N:[0,\infty)\to[0,\infty)$ such that
\begin{equation*}
\begin{aligned}
  |\varphi_\gamma(t,x,u)-\varphi_{\gamma'}(t,x,u)| \le\overline{\omega}_N\left( \mathsf{d} (\gamma, \gamma')\right)
\end{aligned}
\end{equation*}
for any $t\in[0,T], |x|\le N, u \in U, \gamma,\gamma'\in\Gamma$, where $\varphi_\gamma$ represents $b_\gamma, \sigma_\gamma, f_\gamma, h_\gamma$ and its derivatives in $(x,u)$ up to second orders separately.
\end{description}

There are some comments on (H3). This assumption is necessary because, on the one hand, second-order necessary conditions naturally demand more regularity of the coefficients than those required by the stochastic maximum principle. On the other hand, to carry out the weak convergence analysis and obtain the weak limit of the uncertainty measures from the saddle point property, the corresponding integrands, which involve these coefficients, need to be bounded and continuous with respect to the uncertainty parameter. What's more, to perform the Lebesgue-type differentiation theorem following the Clark-Ocone formula, these assumptions are essential to guarantee the differentiability, integrability, and other regularities of the corresponding terms.

Under Assumption (H3), for $u\in\mathcal{U}^\beta $, $\beta \in[2,\infty)$, and $ \mathbb{E}|x_0|^\beta<\infty$,  \eqref{state-equa-gene-main} has a unique solution(cf. \cite[Theorem 3.3.1]{zhjf17}). Besides, it can be verified that 
  (cf.  \cite[Theorem 3.4.3]{zhjf17}): 
\begin{equation}\label{zgjg}
    \bald 
    \sup_{\gamma \in\Gamma }\mathbb{E}\sup_{t\in[0,T]}|x_{\gamma }(t)|^\beta  
&
   \lesssim \sup_{\gamma \in\Gamma } \mathbb{E} \Big[|x_{\gamma }(0)|^\beta +\Big(\int_{0}^{T} |b_\gamma (t,0,u(t))|\dif t\Big)^\beta +\Big(\int_{0}^{T} |\sigma_\gamma  (t,0,u(t))|^2\dif t\Big)^{\frac{\beta}{2}}\Big]  
   \\  & 
   <\infty,
    \eald
\end{equation}
and further for \eqref{cost-function-main}, 
\beq\label{wpec}
\bald 
& \sup_{\gamma \in\Gamma } \Big|\mathbb{E}\Big[\int_0^T f_\gamma(t,x_\gamma(t),u(t))\mathrm{d}t
       +h_\gamma(x_\gamma(T))\Big]\Big|^{\frac{\beta}{2}}
       \\  & \indent 
        \lesssim  \sup_{\gamma \in\Gamma } \Big\{1+\mathbb{E} \int_0^T | u(t) |^\beta \mathrm{d}t +\mathbb{E} \sup_{t\in[0,T]} |x_\gamma(t)|^\beta 
     +\mathbb{E} |x_\gamma(T)|^\beta  \Big\}
       <\infty.
\eald 
\eeq

Denote 
\[v(\cdot):=u(\cdot)-\bar{u}(\cdot),\quad  u^\varepsilon(\cdot):=\bar{u}(\cdot)+\varepsilon v(\cdot), \quad   \delta x_\gamma(\cdot):=x^\varepsilon_\gamma(\cdot)-\bar{x}_\gamma(\cdot),\quad \varepsilon \in(0,1),\]
and for $\varphi_\gamma=b_\gamma, \sigma_\gamma, f_\gamma$, any $\gamma\in\Gamma$,
\begin{equation*}
  \left\{
  \begin{aligned}
    \partial_x \varphi_\gamma(t)&:=\partial_x \varphi_\gamma(t,\bar{x}_\gamma(t),\bar{u}(t)),  \indent\ \  
    \partial_u \varphi_\gamma(t):=\partial_u \varphi_\gamma(t,\bar{x}_\gamma(t),\bar{u}(t)),   
    \\
    \partial_{xx} \varphi_\gamma(t)&:=\partial_{xx} \varphi_\gamma(t,\bar{x}_\gamma(t),\bar{u}(t)),  \indent
    \partial_{xu} \varphi_\gamma(t):=\partial_{xu} \varphi_\gamma(t,\bar{x}_\gamma(t),\bar{u}(t)),   
    \\
    \partial_{uu} \varphi_\gamma(t)&:=\partial_{uu}\varphi_\gamma(t,\bar{x}_\gamma(t),\bar{u}(t)).
  \end{aligned}
  \right.
\end{equation*}

We adopt the classical convex variational method (cf.  \cite{yongjiongminzhouxunyu1999book}) and introduce the following first-order and second-order linearized equations of the state equations~\eqref{state-equa-gene-origi}, parameterized by $\gamma\in\Gamma$:
\begin{equation}\label{variant-equat-first}
\left\{
\begin{aligned}
&\mathrm{d}y_{1,\gamma}(t)=\left[\partial_x b_\gamma(t)y_{1,\gamma}(t)+\partial_u b_\gamma(t)v(t)\right]\mathrm{d}t   +\left[\partial_x\sigma_\gamma(t)y_{1,\gamma}(t)+\partial_u\sigma_\gamma(t)v(t)\right]\mathrm{d}W(t),    \quad   t\in[0,T],       \\
&y_{1,\gamma}(0)=0,
\end{aligned}
\right.
\end{equation}
and
\begin{equation}\label{variant-equat-second}
\left\{
\begin{aligned}
&\mathrm{d}y_{2,\gamma}(t)=\big[\partial_x b_\gamma(t)y_{2,\gamma}(t)  +y_{1,\gamma}(t)^\top\partial_{xx}b_\gamma(t)y_{1,\gamma}(t)  +2v(t)^\top\partial_{xu}b_\gamma(t)y_{1,\gamma}(t)  \\
&\indent\indent\indent  
 +v(t)^\top\partial_{uu}b_\gamma(t)v(t)\big]\mathrm{d}t    +\big[\partial_x\sigma_\gamma(t)y_{2,\gamma}(t) +y_{1,\gamma}(t)^\top\partial_{xx}\sigma_\gamma(t)y_{1,\gamma}(t)  \\
 &\indent\indent\indent
+2v(t)^\top\partial_{xu}\sigma_\gamma(t)y_{1,\gamma}(t) +v(t)^\top\partial_{uu}\sigma_\gamma(t)v(t)\big]\mathrm{d}W(t),    \quad   t\in[0,T],       \\
&y_{2,\gamma}(0)=0.
\end{aligned}
\right.
\end{equation}

\begin{prop}\label{lem-ele-y1y2-delta-x-esti}
Under Assumptions \emph{(H1)} and \emph{(H3)}, for any $\beta\in[1,\infty)$, $u(\cdot), \bar{u}(\cdot)\in \mathcal{U}^\beta$, it yields 
\begin{eqnarray*}
  &&\sup_{\gamma\in\Gamma}\|y_{1,\gamma}\|^\beta_{\infty,\beta}=O(1),\indent \sup_{\gamma\in\Gamma}\|y_{2,\gamma}\|^\beta_{\infty,\beta}=O(1).    
\end{eqnarray*}
Besides, as $ \varepsilon \to 0 $,
\begin{eqnarray*}
     && \sup_{\gamma\in\Gamma}\|\delta x_\gamma\|^\beta_{\infty,\beta}=O(\varepsilon^\beta),   \indent 
    \sup_{\gamma\in\Gamma}\|\delta x_\gamma-\varepsilon y_{1,\gamma}\|^\beta_{\infty,\beta} =O(\varepsilon^{2\beta }), 
    \\  &&  
    \sup_{\gamma\in\Gamma}\big\|\delta x_\gamma-\varepsilon y_{1,\gamma} -\frac{\varepsilon^2}{2}y_{2,\gamma}\big\|^\beta_{\infty,\beta} =O(\varepsilon^{3\beta }).
  \end{eqnarray*}
\end{prop}
\begin{proof}
    The proof is lengthy but straightforward,  so we omit the details.  Similar analyses are commonly found in the literature (cf. \cite{yongjiongminzhouxunyu1999book}). 
\end{proof}

\begin{prop}\label{lem-ele-y1y2-delta-x-esti-sec-f}
Under Assumptions \emph{(H1)} and \emph{(H3)}, as $\varepsilon \to 0$, it holds that 
\begin{equation}\label{est-sup-theta-f-h-ep-six1}
\begin{aligned}
&  \sup_{\gamma\in\Gamma}\Big[ \mathbb{E}\Big|h_\gamma(x_\gamma^\varepsilon(T)) -h_\gamma(\bar{x}_\gamma(T))- \langle\partial_x h_\gamma(\bar{x}_\gamma(T)), \delta x_\gamma(T)\rangle   \\
&\indent\indent\indent   -\frac{1}{2}\left\langle\partial_{xx}h_\gamma(\bar{x}_\gamma(T))\delta x_\gamma(T),\delta x_\gamma(T)\right\rangle \Big|\Big] =O(\varepsilon^{3}),
\end{aligned}
\end{equation}
and 
\begin{align}\label{est-sup-theta-f-h-ep-six2}
&  \sup_{\gamma\in\Gamma}\Big[\mathbb{E} \int_0^T \Big|f_\gamma(t,x_\gamma^\varepsilon(t),u^\varepsilon(t)) -f_\gamma(t,\bar{x}_\gamma(t),\bar{u}(t)) -
 \langle\partial_x f_{\gamma}(t),\delta x_{\gamma}(t)\rangle  \notag  \\
&\indent\indent\indent  -\varepsilon \langle\partial_u f_{\gamma}(t),v(t)\rangle -\varepsilon\langle\partial_{xu}f_{\gamma}(t)\delta x_{\gamma}(t),v(t)\rangle \\
&\indent\indent\indent   -\frac{1}{2}\langle\partial_{xx}f_\gamma(t)\delta x_\gamma(t),\delta x_{\gamma}(t)\rangle
-\frac{{\varepsilon}^2}{2}\langle\partial_{uu}f_{\gamma}(t)v(t),v(t)\rangle\Big|\mathrm{d}t  \Big] =O(\varepsilon^{3}).  \notag 
\end{align}
\end{prop}

\begin{proof}
It directly follows from the Assumption (H3), together with the estimations in Proposition~\ref{lem-ele-y1y2-delta-x-esti}.
\end{proof}

\begin{prop}\label{pr-su-co-li-0-co}
Let Assumptions \emph{(H3)} and \emph{(H4)} hold, then for $u\in \mathcal U^\beta , \beta \ge 4$, we have 
\begin{equation}\label{lem-cost-continuous-theta-theta-pron-descri}
\begin{aligned}
&\lim_{ \ell \to0} \sup_{ \mathsf{d} (\gamma,\gamma')\le \ell } \Big\{ \mathbb{E}\sup_{t\in[0,T]}\Big(\left|x_\gamma(t)-x_{\gamma'}(t)\right|^2 +\left|h_\gamma(x_\gamma(T))-h_{\gamma'}(x_{\gamma'}(T))\right|^2  \\
&\indent\indent\indent  +\left|y_{1,\gamma}(t)-y_{1,\gamma'}(t)\right|^2 +\Big|\int_t^T[f_\gamma(s,x_\gamma(s),u(s))-f_{\gamma'}(s,x_{\gamma'}(s),u(s))]\mathrm{d}s\Big|^2\Big) \Big\} =0.
\end{aligned}
\end{equation}
\end{prop}
\begin{proof}
The results come from the integrable of $x_{\gamma}, x_{\gamma'}$, $y_{1,\gamma}, y_{1,\gamma'}$, $u$, and (H4). 
We sketch the proof of the first term for readers' convenience. 
The process $x_\gamma-x_{\gamma'}$ fulfills the following equation
\begin{equation}\label{smch}
    \left\{
    \begin{aligned}
    &\mathrm{d}(x_\gamma-x_{\gamma'})(t)=\big[\int_{0}^{1}\partial_x b_{\gamma'}(t,x_{\gamma'} + \theta (x_\gamma-x_{\gamma'}) ,u)\mathrm{d}\theta (x_\gamma-x_{\gamma'}) + \psi^b_{\gamma ,\gamma' }  \big]\mathrm{d}t
    \\  & \indent  \indent  \indent  \indent \quad 
        +\big[\int_{0}^{1}\partial_x \sigma_{\gamma'}(t,x_{\gamma'} + \theta (x_\gamma-x_{\gamma'}) ,u)\mathrm{d}\theta (x_\gamma-x_{\gamma'}) + \psi^\sigma_{\gamma ,\gamma' }  \big]\mathrm{d}W(t),    \quad   t\in[0,T],      \\
    &(x_\gamma-x_{\gamma'})(0)=0,
    \end{aligned}
    \right.
\end{equation}
where $\psi^\varphi_{\gamma ,\gamma' }=\varphi_\gamma (t,x_\gamma ,u) - \varphi_{\gamma'} (t,x_\gamma ,u)$ for $\varphi = b, \sigma $. 
Utilizing \cite[Theorem 3.4.3]{zhjf17}, it derives that 
\begin{equation}\label{zcgx}
    \bald 
    \mathbb{E}\sup_{t\in[0,T]}|(x_\gamma-x_{\gamma'})(t)|^2  
   \lesssim  \mathbb{E} \Big[ \Big(\int_{0}^{T} |\psi^b_{\gamma ,\gamma' }|\dif t\Big)^2 +\Big(\int_{0}^{T} |\psi^\sigma_{\gamma ,\gamma' }|^2\dif t\Big)\Big].   
    \eald
\end{equation} 
By (H3) and (H4), it holds that for any $\gamma', \gamma \in \Gamma  $, a.e.  a.s. 
\begin{align*}
    |\psi^b_{\gamma ,\gamma' }| & =|b_\gamma (t,x_\gamma ,u) - b_{\gamma'} (t,x_\gamma ,u)|
\\
& \le \overline{\omega}_N\left( \mathsf{d} (\gamma, \gamma')\right) 1_{\{| {x}_{\gamma}|\le N\}} 
+ |b_{\gamma}(t, {x}_{\gamma} , {u}  )-b_{\gamma'}(t, {x}_{\gamma} , {u}  )| 1_{\{| {x}_{\gamma}| > N\}}.
\end{align*}
Then for any fixed $N>0$, it yields 
\begin{align*}
    \lim_{ \ell \to0} \sup_{ \mathsf{d} (\gamma,\gamma')\le \ell }\mathbb{E} \Big(\int_{0}^{T} |\psi^b_{\gamma ,\gamma' }|\dif t\Big)^2 \lesssim \sup_{\gamma \in \Gamma }\mathbb{E}\Big(\int_{0}^{T} (1+|x_\gamma |)\cdot \frac{|x_\gamma |}{N} \dif t \Big)^2 . 
\end{align*}
Recalling \eqref{zgjg} and sending $N\to \infty$, we obtain 
$\lim_{ \ell \to0} \sup_{ \mathsf{d} (\gamma,\gamma')\le \ell }\mathbb{E} \big(\int_{0}^{T} |\psi^b_{\gamma ,\gamma' }|\dif t\big)^2=0$. 
Similarly, it can be deduced that 
$\lim_{ \ell \to0} \sup_{ \mathsf{d} (\gamma,\gamma')\le \ell } \mathbb{E}  \int_{0}^{T} |\psi^\sigma_{\gamma ,\gamma' }|^2\dif t =0$.
Then by \eqref{zcgx}, we derive 
\[\lim_{ \ell \to0} \sup_{ \mathsf{d} (\gamma,\gamma')\le \ell } \mathbb{E}\sup_{t\in[0,T]}|(x_\gamma-x_{\gamma'})(t)|^2 =0. \]

The estimates for the other three parts in \eqref{lem-cost-continuous-theta-theta-pron-descri} are similar (cf.  \eqref{enpcn1}-\eqref{enpcn4}). 
\end{proof}

To utilize the dual analysis methods to study the optimal control problems, we consider two adjoint equations parameterized by $\gamma\in\Gamma$:
\begin{equation}\label{adjoint-equa-first}
\left\{
\begin{aligned}
\mathrm{d}P_{1,\gamma}(t)&=-\Big[\partial_x b_\gamma(t)^\top P_{1,\gamma}(t) +\partial_x \sigma_\gamma(t)^\top Q_{1,\gamma}(t)-\partial_x f_\gamma(t)\Big]\mathrm{d}t     +Q_{1,\gamma}(t)\mathrm{d}W(t),    \quad   t\in[0,T],       \\
P_{1,\gamma}(T)&=-\partial_x h_\gamma(\bar{x}_\gamma(T)),
\end{aligned}
\right.
\end{equation}
and
\begin{equation}\label{adjoint-equa-second}
\left\{
\begin{aligned}
&\mathrm{d}P_{2,\gamma}(t)=-\Big[\partial_x b_\gamma(t)^\top P_{2,\gamma}(t)
+P_{2,\gamma}(t)\partial_x b_\gamma(t)
+\partial_x \sigma_\gamma(t)^\top P_{2,\gamma}(t)\partial_x \sigma_\gamma(t)  +\partial_x \sigma_\gamma(t)^\top Q_{2,\gamma}(t)  \\
&\indent\indent\indent\quad   
+Q_{2,\gamma}(t)\partial_x \sigma_\gamma(t)
+\partial_{xx}H_{\gamma}(t)\Big]\mathrm{d}t    +Q_{2,\gamma}(t)\mathrm{d}W(t),    \quad   t\in[0,T],       \\
&P_{2,\gamma}(T)=-\partial_{xx} h_\gamma(\bar{x}_\gamma(T)),
\end{aligned}
\right.
\end{equation}
where $\partial_{xx}H_{\gamma}(t)=\partial_{xx}H_{\gamma} (t,\bar{x}_\gamma(t),\bar{u}(t),P_{1,\gamma}(t),Q_{1,\gamma}(t))$, and $ H_{\gamma}$ is   the Hamiltonian   defined by 
\begin{equation}\label{hamiltonian}
  H_{\gamma}(t,x_\gamma,u,p_\gamma,q_\gamma):=\langle p_\gamma,b_\gamma(t,x_\gamma,u)\rangle +\langle q_\gamma,\sigma_\gamma(t,x_\gamma,u)\rangle  -f_\gamma(t,x_\gamma,u), \quad  \gamma\in\Gamma,
\end{equation}
where $(t,x_\gamma,u,p_\gamma,q_\gamma)\in[0,T]\times\mathbb{R}^n\times U\times\mathbb{R}^n\times\mathbb{R}^n$.

\begin{prop}
Under   Assumptions \emph{(H1)} and \emph{(H3)},  for any $\beta\in[1,\infty)$, $\bar{u}(\cdot)\in \mathcal{U}^\beta$, there uniquely exists \\ \indent
$(P_{1,\gamma}(\cdot),Q_{1,\gamma}(\cdot))$ $\in$ $ L_\mathbb{F}^\beta (\Omega; C([0,T]; \mathbb{R}^n)) $ $\times $ $L_\mathbb{F}^\beta (\Omega; L^2(0,T; \mathbb{R}^n))$, \\ \indent
$(P_{2,\gamma}(\cdot),Q_{2,\gamma}(\cdot))$ $\in $ $L_\mathbb{F}^\beta (\Omega; C([0,T]; \mathbf{S}^n)) $  $\times$ $L_\mathbb{F}^\beta(\Omega;L^2(0,T;\mathbf{S}^n)),$  \\
satisfying \eqref{adjoint-equa-first} and \eqref{adjoint-equa-second} separatly, 
and 
\begin{equation}\label{aeqpq1}
\begin{aligned}
\sup_{\gamma\in\Gamma} \mathbb{E}\Big(\sup_{t\in[0,T]}\left|P_{1,\gamma}(t) \right|^\beta  +\sup_{t\in[0,T]}\left|P_{2,\gamma}(t) \right|^\beta \Big) <\infty,
\end{aligned}
\end{equation}
\begin{equation}\label{aeqpq2}
  \begin{aligned}
    \sup_{\gamma\in\Gamma} \mathbb{E}\Big(\int_0^T\left|Q_{1,\gamma}(t) \right|^2 \mathrm{d}t \Big)^\frac{\beta }{2} + \sup_{\gamma\in\Gamma} \mathbb{E}\Big(\int_0^T\left|Q_{2,\gamma}(t)  \right|^2 \mathrm{d}t\Big)^\frac{\beta }{2}<\infty.
  \end{aligned}
  \end{equation}
Besides,  given Assumptions  \emph{(H1)}, \emph{(H3)} and \emph{(H4)},  
\begin{equation}\label{par-cont-p1p2-de}
\begin{aligned}
\lim_{ \ell \to0} \sup_{ \mathsf{d} (\gamma,\gamma')\le \ell } \Big[\mathbb{E}\Big(\sup_{t\in[0,T]}\left|P_{1,\gamma}(t)-P_{1,\gamma'}(t)\right|^2 +\sup_{t\in[0,T]}\left|P_{2,\gamma}(t)-P_{2,\gamma'}(t)\right|^2\Big)\Big]=0
\end{aligned}
\end{equation}
and
\begin{equation}\label{par-cont-q1q2-de}
\begin{aligned}
\lim_{ \ell \to0} \sup_{ \mathsf{d} (\gamma,\gamma')\le \ell } \Big[\mathbb{E}\int_0^T\Big(\left|Q_{1,\gamma}(t)-Q_{1,\gamma'}(t)\right|^2 + \left|Q_{2,\gamma}(t)-Q_{2,\gamma'}(t)\right|^2 \Big)\mathrm{d}t\Big]=0.
\end{aligned}
\end{equation}
\end{prop}
\begin{proof}
The existence and uniqueness for \eqref{adjoint-equa-first} and \eqref{adjoint-equa-second}, and the proof of \eqref{aeqpq1} and \eqref{aeqpq2}, all of which can be found in \cite{kpq97}, by the Burkholder-Davis-Gundy inequality, H\"{o}lder's inequality, and others.  
We only give the detailed proof about $(P_{1,\gamma}(\cdot),Q_{1,\gamma}(\cdot))$. That for $(P_{2,\gamma}(\cdot),Q_{2,\gamma}(\cdot))$ are similar. 

By \cite[Theorem 4.4.4]{zhjf17}, for $\beta \ge 2$, it can be shown that
    \begin{align}
\sup_{\gamma\in\Gamma} \mathbb{E} \sup_{t\in[0,T]}\left|P_{1,\gamma}(t) \right|^\beta  + \sup_{\gamma\in\Gamma} \mathbb{E}\Big(\int_0^T\left|Q_{1,\gamma}(t) \right|^2 \mathrm{d}t \Big)^\frac{\beta }{2}   \notag 
&
   \lesssim  \E |\! -\partial_x h_\gamma(\bar{x}_\gamma(T))|^\beta + \E \Big(\int_{0}^{T}|\! -\partial_x f_\gamma(t)| \dif t\Big)^\beta
   \\  & 
   <\infty . 
\end{align}
Moreover, it holds that 
\begin{equation}\label{pcPQ1m}
    \bald
    &  \mathbb{E} \sup_{t\in[0,T]}\left|P_{1,\gamma}(t) - P_{1,\gamma'}(t) \right|^\beta  +   \mathbb{E}\Big(\int_0^T\left|Q_{1,\gamma}(t)-Q_{1,\gamma'}(t) \right|^2 \mathrm{d}t \Big)^\frac{\beta }{2}
   \\  & \indent
   \lesssim  \E |\partial_x h_{\gamma'}(\bar{x}_{\gamma'}(T))-\partial_x h_\gamma(\bar{x}_\gamma(T))|^\beta + \E \Big(\int_{0}^{T}|\partial_x f_{\gamma'}(t)-\partial_x f_\gamma(t)| \dif t\Big)^\beta. 
    \eald
\end{equation}
For the last term, by Assumptions (H3) and (H4), interpolating one term  $\partial_x f_\gamma(t,\overline{x}_{\gamma'} ,\overline{u})$, it follows that for any $\gamma', \gamma \in \Gamma  $, a.e.  a.s. 
\begin{equation}\label{enpcn1}
\bald 
&  |\partial_x f_{\gamma'}(t)-\partial_x f_\gamma(t)| 
\\  & \indent 
\le |\partial_x f_{\gamma}(t,\overline{x}_\gamma ,\overline{u}  )-\partial_x f_\gamma(t,\overline{x}_{\gamma'} ,\overline{u})| 
+ |\partial_x f_{\gamma'}(t,\overline{x}_{\gamma'} ,\overline{u}  )-\partial_x f_\gamma(t,\overline{x}_{\gamma'} ,\overline{u})|  
\\  & \indent 
\lesssim |\overline{x}_{\gamma}-\overline{x}_{\gamma'}| 
+ \overline{\omega}_N\left( \mathsf{d} (\gamma, \gamma')\right) 1_{\{|\overline{x}_{\gamma'}|\le N\}} 
+ |\partial_x f_{\gamma}(t,\overline{x}_{\gamma'} ,\overline{u}  )-\partial_x f_{\gamma'}(t,\overline{x}_{\gamma'} ,\overline{u}  )| 1_{\{|\overline{x}_{\gamma'}| > N\}}. 
\eald 
\end{equation}
Then it further derives 
\begin{equation}\label{enpcn2}
    \bald 
    \E \Big(\int_{0}^{T}|\partial_x f_{\gamma'}(t)-\partial_x f_\gamma(t)| \dif t\Big)^\beta 
  &  
   \lesssim \E \Big(\int_{0}^{T} |\overline{x}_{\gamma}-\overline{x}_{\gamma'}| \dif t\Big)^\beta  
   +\E \Big(\int_{0}^{T}\overline{\omega}^\beta_N\left( \mathsf{d} (\gamma, \gamma')\right) 1_{\{|\overline{x}_{\gamma'}|\le N\}} \dif t\Big)      \\  & \quad 
   + \E \Big(\int_{0}^{T}  (1+|\overline{x}_{\gamma'}|+|\bar u|)\cdot \frac{|\overline{x}_{\gamma'}|}{N}  \dif t\Big)^\beta .
    \eald 
\end{equation}
For any fixed  $N>0$, it can be concluded that 
\begin{equation}\label{enpcn3}
    \bald 
   & \lim_{ \ell \to0} \sup_{ \mathsf{d} (\gamma,\gamma')\le \ell } \E \Big(\int_{0}^{T}|\partial_x f_{\gamma'}(t)-\partial_x f_\gamma(t)| \dif t\Big)^\beta  
   \lesssim  \sup_{\gamma'\in\Gamma} \Big[\E \Big(1+ \sup_{t\in[0,T]}  |\overline{x}_{\gamma'}|^{2\beta}+\int_{0}^{T}|\bar u|^{2\beta} \dif t\Big)\Big]  N^{-\beta }.
    \eald 
\end{equation}
Letting $N\to \infty$, we see 
\begin{equation}\label{enpcn4}
    \bald 
   & \lim_{ \ell \to0} \sup_{ \mathsf{d} (\gamma,\gamma')\le \ell } \E \Big(\int_{0}^{T}|\partial_x f_{\gamma'}(t)-\partial_x f_\gamma(t)| \dif t\Big)^\beta  =0. 
    \eald 
\end{equation}
Similarly,  
\begin{equation*}
    \bald 
   &  \lim_{ \ell \to0} \sup_{ \mathsf{d} (\gamma,\gamma')\le \ell } \E |\partial_x h_{\gamma'}(\bar{x}_{\gamma'}(T))-\partial_x h_\gamma(\bar{x}_\gamma(T))|^\beta =0 . 
    \eald 
\end{equation*}
Recalling \eqref{pcPQ1m}, we obtain 
\begin{equation*}
    \bald 
   &  \lim_{ \ell \to0} \sup_{ \mathsf{d} (\gamma,\gamma')\le \ell } \Big\{ \mathbb{E} \sup_{t\in[0,T]}\left|P_{1,\gamma}(t) - P_{1,\gamma'}(t) \right|^2  +  \mathbb{E}\Big(\int_0^T\left|Q_{1,\gamma}(t)-Q_{1,\gamma'}(t) \right|^2 \mathrm{d}t \Big) \Big\}=0.
    \eald 
\end{equation*}
\end{proof}

Denote 
\begin{equation}\label{defi-S}
\begin{aligned}
& \mathbb{S}_\gamma(t,x_\gamma,u,p_{1,\gamma},q_{1,\gamma},p_{2,\gamma},q_{2,\gamma}):=  \partial_{xu}H_{\gamma}(t,x_\gamma,u,p_{1,\gamma},q_{1,\gamma}) +\partial_u b_\gamma(t,x_\gamma,u)^\top p_{2,\gamma}
     \\
&\indent\indent\indent\indent\indent\indent +\partial_u \sigma_\gamma(t,x_\gamma,u)^\top q_{2,\gamma}    +\partial_u \sigma_\gamma(t,x_\gamma,u)^\top p_{2,\gamma}\partial_x \sigma_\gamma(t,x_\gamma,u),
\end{aligned}
\end{equation}
where $(t,x_\gamma,u,p_{1,\gamma},q_{1,\gamma},p_{2,\gamma},q_{2,\gamma})\in[0,T]\times \mathbb{R}^n\times U\times \mathbb{R}^n\times \mathbb{R}^n\times\mathbf{S}^n\times\mathbf{S}^n$.
And denote
\begin{equation}\label{defi-St}
\mathbb{S}_\gamma(t):=\mathbb{S}_\gamma(t,\bar{x}_\gamma(t),\bar{u}(t), P_{1,\gamma}(t),Q_{1,\gamma}(t),P_{2,\gamma}(t),Q_{2,\gamma}(t)),\quad  t\in[0,T],
\end{equation}
of which the following properties are needed later. 

\begin{prop}\label{pr-su-co-listhet}
  Provided with Assumptions  \emph{(H1)}, \emph{(H3)} and \emph{(H4)},  and $\bar{u}(\cdot)\in \mathcal{U}^\beta$, $\beta \ge 2$, we have 
  \begin{equation}\label{essthete}
    \begin{aligned}
     \sup_{\gamma\in\Gamma} \mathbb{E}\int_{0}^{T} |\mathbb{S}_\gamma(t)|^2 \mathrm{d}t <\infty.
\end{aligned}
\end{equation}
Besides,
\begin{equation}\label{lethetheprocontisup}
  \begin{aligned}
  &\lim_{\ell \to0} \sup_{ \mathsf{d} (\gamma,\gamma')\le\ell }  \mathbb{E}\int_{0}^{T} \left|\mathbb{S}_\gamma(t)-\mathbb{S}_{\gamma'}(t)\right|^{2-}\mathrm{d}t =0.
\end{aligned}
\end{equation}
\end{prop}
\begin{proof}
    Utilizing \eqref{aeqpq1} and \eqref{aeqpq2}, together with (H3), it yields that 
    \begin{align*}
    \sup_{\gamma\in\Gamma} \mathbb{E}\int_{0}^{T} |\mathbb{S}_\gamma(t)|^2 \mathrm{d}t 
    & \lesssim  \sup_{\gamma\in\Gamma} \mathbb{E}\int_{0}^{T}\big( |\partial_{xu}H_{\gamma}(t,\bar{x}_\gamma,\bar{u},P_{1,\gamma},Q_{1,\gamma})|^2 +| \partial_u b_\gamma(t,\bar{x}_\gamma,\bar{u})^\top P_{2,\gamma}|^2  \\
     & \indent\indent\indent\quad +| \partial_u \sigma_\gamma(t,\bar{x}_\gamma,\bar{u})^\top Q_{2,\gamma} |^2    +| \partial_u \sigma_\gamma(t,\bar{x}_\gamma,\bar{u})^\top P_{2,\gamma}\partial_x \sigma_\gamma(t,\bar{x}_\gamma,\bar{u}) |^2\big) \mathrm{d}t  \\
    &  \lesssim
     \sup_{\gamma\in\Gamma} \mathbb{E}\int_{0}^{T} \big(1+|P_{1,\gamma}(t)|^2 +|P_{2,\gamma}(t)|^2 +|Q_{1,\gamma}(t)|^2 +|Q_{2,\gamma}(t)|^2 \big) \mathrm{d}t   <\infty.
\end{align*}
By the procedure described in \eqref{enpcn1}-\eqref{enpcn4}, it can be proved that 
\begin{equation}\label{lgxzy}
\begin{aligned}
    &\lim_{\ell \to0} \sup_{ \mathsf{d} (\gamma,\gamma')\le\ell }  \mathbb{E}\int_{0}^{T} \left|\mathbb{S}_\gamma(t)-\mathbb{S}_{\gamma'}(t)\right|\mathrm{d}t =0.
\end{aligned}
\end{equation}
Furthermore, together with \eqref{essthete}, utilizing H\"{o}lder's inequality repeatedly, we obtain \eqref{lethetheprocontisup}. 
\end{proof}

\begin{rem}
    In Theorem \ref{thmfirornedes} about the integral type second-order necessary conditions, the estimate \eqref{lgxzy} in $L^1(\mathbb{P} (\mathrm d \omega )\times \mathrm d t)$ sense is sufficient to carry out the proof. And actually \eqref{lgxzy} can be strengthened to \eqref{lethetheprocontisup} provided with  \eqref{essthete}. 
    While in the pointwise setting (see Theorem \ref{thm-point-nece-cla} below), we have more restrictions \emph{(H6)} on coefficients to guarantee the existence of certain Malliavin derivative and the application of Clark-Ocone theorem.
    Then additionally with this condition, the following \eqref{jhzls} can be proved with a direct estimate 
\begin{equation}\label{jhzls}
        \begin{aligned}
        &\lim_{\ell \to0} \sup_{ \mathsf{d} (\gamma,\gamma')\le\ell }  \mathbb{E}\int_{0}^{T} \left|\mathbb{S}_\gamma(t)-\mathbb{S}_{\gamma'}(t)\right|^{2}\mathrm{d}t =0. 
      \end{aligned}
\end{equation}
\end{rem}

As a consequence  of \eqref{lgxzy}, using approximation method with the help of partitions of unity theorem in Lemma \ref{partunith}, 
$\mathbb{S}$ is progressively measurable in the sense that for any $t\in[0,T]$,
\begin{align*}
&\mathbb{S}:\quad \Gamma\times\Omega\times[0,t]\times \mathbb{R}^n\times U\times \mathbb{R}^n\times\mathbb{R}^n\times\mathbf{S}^n\times\mathbf{S}^n\to\mathbb{R}^{m\times n};  \\
&\mathcal{B}(\Gamma)\otimes\mathcal{F}_t\otimes\mathcal{B}([0,t])\otimes \mathcal{B}(\mathbb{R}^n)\otimes \mathcal{B}(\mathbb{R}^m)\otimes \mathcal{B}(\mathbb{R}^n)\otimes\mathcal{B}(\mathbb{R}^n)\otimes\mathcal{B}(\mathbf{S}^n)\otimes\mathcal{B}(\mathbf{S}^n)\to\mathcal{B}(\mathbb{R}^{m\times n})
\end{align*}
is measurable.
We sketch the proof here for the convenience of the readers. Similar analyses hold for $x,y_1, (P_1,Q_1), (P_2,Q_2), \mathcal{D}_\cdot\mathbb{S}_\cdot(\cdot)$, and the details are omitted. 

Since $\Gamma $ is a locally compact Polish space, for each $N\in \mathbb{N}_+$, choose a compact subset $\mathsf{S}^N \subset \Gamma $ with $\lambda (\gamma \not\in \mathsf{S}^N)<\frac{1}{N} $. 
Besides, there exists a sequence of open neighborhoods $\{B(\gamma_i,\frac{1}{2N} )\}_{i=1}^{\varsigma_N}$ such that $\mathsf{S}^N \subset \bigcup_{i=1}^{\varsigma_N} B(\gamma_i,\frac{1}{2N} )$. 
By partitions of unity (Lemma \ref{partunith}), there exists a sequence of continuous functions $\rho_i: \Gamma \to [0,1]\subset \mathbb{R} $ such that $\rho_i(\gamma )=0$ for $\gamma \not\in B(\gamma_i,\frac{1}{2N} ), i=1,\cdots,\varsigma_N$, and $\sum_{i = 1}^{\varsigma_N}\rho_i(\gamma )=1  $ for $\gamma \in \mathsf{S}^N$.
Furthermore, choose the following  joint measurable process $ {P}^N_\gamma(\cdot)$ with $\gamma_i^*$ satisfying $\rho_i(\gamma_i^*)>0$: 
\[ {\mathbb{S}}^N_\gamma (t):= \sum_{i = 1}^{\varsigma_N} \mathbb{S}_{\gamma_i^*}(t) \rho_i(\gamma )\chi_{\{\gamma \in \mathsf{S}^N\}}.\] 
It is sufficient to prove 
\begin{align}\label{lcyPgc2} 
\lim_{N\to\infty}\int_\Gamma \int_{0}^{T} \mathbb{E}\ | \mathbb{S}^N_\gamma (t)-\mathbb{S}_{\gamma}(t) | \mathrm{d}t \lambda (\mathrm{d}\gamma ) =0 .
\end{align}
Indeed, 
\begin{align*}
 \int_{0}^{T} \mathbb{E}\ | \mathbb{S}^N_\gamma (t)-\mathbb{S}_{\gamma}(t) | \dif t 
    & 
    \le 
    \sum_{i = 1}^{\varsigma_N} \int_{0}^{T} \mathbb{E}\ | \mathbb{S}^N_{\gamma_i^*} (t)-\mathbb{S}_{\gamma}(t) |  \dif t \rho_i(\gamma )\chi_{\{\gamma \in \mathsf{S}^N\}}
    + \int_{0}^{T} \mathbb{E}\ | \mathbb{S}_{\gamma}(t) |  \dif t  \chi_{\{\gamma \not\in \mathsf{S}^N\}}
    \\
    &  
    \le \sup_{\mathsf{d}(\gamma_1 ,\gamma_2)\le \frac{1}{N}  }  \int_{0}^{T} \mathbb{E}\ | \mathbb{S}_{\gamma_1}(t)-\mathbb{S}_{\gamma_2}(t) |  \dif t
    + \sup_{\gamma \in \Gamma } \int_{0}^{T} \mathbb{E}\ |\mathbb{S}_{\gamma}(t)|  \dif t \chi_{\{\gamma \not\in \mathsf{S}^N\}}. 
\end{align*}
Then 
\begin{align*}
    \lim_{N\to\infty}\int_\Gamma \int_{0}^{T} \mathbb{E}\ | \mathbb{S}^N_\gamma (t)-\mathbb{S}_{\gamma}(t) | \dif t \lambda (\dif\gamma ) 
&
\lesssim \lim_{N\to\infty} \Big\{ \sup_{\mathsf{d}(\gamma_1 ,\gamma_2)\le \frac{1}{N}  }  \int_{0}^{T} \mathbb{E}\ | \mathbb{S}_{\gamma_1}(t)-\mathbb{S}_{\gamma_2}(t) |  \dif t + \frac{C}{N} \Big\}
\\
&
 \overset{\eqref{lgxzy}}{=}0. 
\end{align*}

Denote
\[\Lambda^{u}:=
       \Big\{\lambda \in\Lambda \  \big|\  J(u(\cdot))=\int_{\Gamma}\mathbb{E} \big[\int_0^T f_\gamma(t,x_\gamma(t),u(t))\mathrm{d}t
       +h_\gamma(x_\gamma(T))\big]\lambda(\mathrm{d}\gamma)\Big\}.\]
Under Assumptions (H1)-(H4), together with \eqref{cost-function-main}, \eqref{wpec}, and \eqref{lem-cost-continuous-theta-theta-pron-descri}, and utilizing the procedure of weak convergence arguments, the set $\Lambda^{u}$ defined above is nonempty for any $ u(\cdot)\in \mathcal{U}^\beta $. Moreover, $\Lambda^{u}$ is convex and weakly compact for any $u(\cdot)\in \mathcal{U}^\beta $. It can be regarded as a corollary of \cite[Lemma 3.5]{humingshangwangfalei2020siam}, where a standard proof in a more general setting is given. Interested readers can also refer to \cite[(21)]{hlw23} for details. The specifics are omitted for conciseness.

The equation \eqref{variant-equat-first} has explicit solution formula due to its linear structure (cf.  \cite[Section 3.1]{zhjf17}).
Denote by $\Phi_\gamma$ the solutions of the matrix-valued stochastic differential equation
\begin{equation*}
\left\{
\begin{aligned}
&\mathrm{d}\Phi_\gamma(t)=\partial_x b_\gamma(t)\Phi_\gamma(t)\mathrm{d}t +\partial_x \sigma_\gamma(t)\Phi_\gamma(t)\mathrm{d}W(t), \quad  t\in[0,T],  \\
&\Phi_\gamma(0)=I_{n\times n}.
\end{aligned}
\right.
\end{equation*}
It  can be  verified that
\begin{equation}\label{solu-form-ele-phi}
  \Phi_\gamma(t)=\Phi_\gamma(\tau)+\int_\tau^t\partial_x b_\gamma(s)\Phi_\gamma(s)\mathrm{d}s+\int_\tau^t\partial_x \sigma_\gamma(s)\Phi_\gamma(s)\mathrm{d}W(s),
\end{equation}
and the solution to \eqref{variant-equat-first} has the following representation:
\begin{equation}\label{lirephy}
\begin{aligned}
y_{1,\gamma}(t)=\Phi_\gamma(t) \! \int_0^t {\Phi_\gamma(s)}^{-1} \! \left(\partial_u b_\gamma(s) \! - \! \partial_x\sigma_\gamma(s)\partial_u\sigma_\gamma(s)\right)v(s)\mathrm{d}s   \!  + \! \Phi_\gamma(t) \! \int_0^t{\Phi_\gamma(s)}^{-1}\partial_u\sigma_\gamma(s)v(s)\mathrm{d}W(s).
\end{aligned}
\end{equation}

The definition of singular involved in this article is as follows.
\begin{defi}\label{singularoptimalcontroldef}

(i)
$\tilde{u}\in\mathcal{U}^\beta $ is called a \emph{singular control in the classical sense} if for any $ \lambda\in\Lambda^{\tilde{u}}$, $\tilde{u}$ satisfies
\begin{equation}\label{def-singular-con}
\left\{
\begin{aligned}
&\int_\Gamma\partial_u H_\gamma(t,\tilde{x}_\gamma(t),\tilde{u}(t),\tilde{P}_{1,\gamma}(t),\tilde{Q}_{1,\gamma}(t))\lambda(\mathrm{d}\gamma)=0 \quad    a.s.\  a.e.\   t\in[0,T],
\\&\int_\Gamma\Big[\partial_{uu}H_{\gamma}(t,\tilde{x}_\gamma(t),\tilde{u}(t),\tilde{P}_{1,\gamma}(t),\tilde{Q}_{1,\gamma}(t)) \\
&\indent +\partial_u \sigma_\gamma(t,\tilde{x}_\gamma(t),\tilde{u}(t))^\top \tilde{P}_{2,\gamma}(t) \partial_u \sigma_\gamma(t,\tilde{x}_\gamma(t),\tilde{u}(t))\Big]\lambda(\mathrm{d}\gamma) =0  \quad   a.s.\  a.e.\   t\in[0,T],
\end{aligned}
\right.
\end{equation}
where $\tilde{x}_\gamma$ is the state with respect to $\tilde{u}$, and $(\tilde{P}_{1,\gamma},\tilde{Q}_{1,\gamma})$, $(\tilde{P}_{2,\gamma},\tilde{Q}_{2,\gamma})$ are given by~\eqref{adjoint-equa-first} and~\eqref{adjoint-equa-second} respectively but with $(\bar{u},\bar{x}_\gamma)$ replaced by $(\tilde{u},\tilde{x}_\gamma)$.

(ii) If the singular control $\tilde{u}$ is also optimal, then it is called a \emph{singular optimal control in the classical sense}.
\end{defi}

\begin{rem}
Note that Definition~\ref{singularoptimalcontroldef} of singular optimal control actually corresponds to the degeneration of the first-order necessary conditions for stochastic optimal control problems, a detailed justification for which can be found in \cite[Remark 3.4]{zhanghaisenzhangxu2015siam}. In other words, it describes the case where stochastic maximum principle provides only trivial information.
\end{rem}

Up to now, we have established the well-posedness of the stochastic optimal control problems formulated in \eqref{state-equa-gene-main}-\eqref{cpmc}. 
Denote 
\begin{equation}\label{notaip}
        \mathfrak{I} (v_1,v_2;\lambda )
        :=  - \int_{\Gamma} \mathbb{E}\int_0^T \left\langle\mathbb{S}_\gamma(t)y_{1,\gamma}(t;v_1), v_2(t)\right\rangle \mathrm{d}t \lambda (\mathrm{d}\gamma);  
        \quad 
        \mathfrak{I} (v;\cdot) := \mathfrak{I} (v,v;\cdot) .
\end{equation}  
We see that $\mathfrak{I} (\cdot;\lambda )$ is nonhomogeneous bi-linear with respect to the perturbed control $u$.

The monotonicity condition plays a fundamental role in the context of optimization theory.
Inspired by Kachurovskii's theorem (cf. \cite{Showalter}), we can strengthen the results in \eqref{initial-second-necessary-condition-thm} to derive a result with common reference measure under the following monotonicity condition:
\begin{description}
  \item[(H5)] 
For any $u_1(\cdot), u_2(\cdot) \in L_\mathbb{F}^4 (\Omega; L^4 ([0,T]; \mathbb{R}^m))$, 
it holds that for any $\gamma \in \Gamma$, 
\[ - \mathbb{E}\int_0^T \left\langle\mathbb{S}_\gamma(t)y_{1,\gamma}(t;u_1-u_2), u_1(t)-u_2(t)\right\rangle \mathrm{d}t   \ge 0.\]
\end{description}
\noindent 
Besides, Malliavin calculus facilitates the application of the Lebesgue type differentiation theorem to obtain the second-order pointwise necessary conditions for stochastic optimal control, which transforms the It\^{o}-Lebesgue integration into multiple Lebesgue integration. However, it works at the expense of additional conditions to handle the new emerging terms. We further present the following technical assumptions:
\begin{description}
  \item[(H6)]
 (i)  $\bar{u}(\cdot )\in \mathbb{L}_{2,\mathbb{F}}^{1,2}(\mathbb{R}^m)$, $\mathbb{S}_\gamma(\cdot)\in \mathbb{L}_{2,\mathbb{F}}^{1,2}(\mathbb{R}^{m\times n})\bigcap L^\infty ([0,T]\times\Omega;\mathbb{R}^{m\times n})$,  \\ \quad 
 and $\mathcal{D}_\cdot\mathbb{S}_\gamma(\cdot) \in L^2([0,T];L^\infty ([0,T]\times\Omega;\mathbb{R}^{m\times n}))$,  
 uniformly for $\gamma\in\Gamma$.
\vspace*{-0.3cm}

\begin{equation}\label{der-phi-nab-phi-cont-ass}
    \hspace*{-2.8cm}
\textnormal{(ii)} \ \  
\begin{aligned}
\lim_{\ell\to0}\sup_{ \mathsf{d} (\gamma,\gamma')\le\ell } \int_0^T \! \int_0^T\mathbb{E}   \left|\mathcal{D}_s\mathbb{S}_\gamma(t)-\mathcal{D}_s\mathbb{S}_{\gamma'}(t)\right|^2 \mathrm{d}s\mathrm{d}t =0.   \hspace*{3cm}
\end{aligned}
\end{equation}
\vspace*{-0.4cm}
\end{description}

In the following theorem, we present the first main result of this paper which characterizes the integral type second-order necessary condition for the singular stochastic optimal  control problem in \eqref{state-equa-gene-main}-\eqref{cpmc}.

\begin{thm}\label{thmfirornedes}
Let Assumptions \emph{(H1)-(H4)} hold, and $\bar{u}(\cdot)$  a singular optimal control as defined in Definition \ref{singularoptimalcontroldef}, satisfying $\bar{u}(\cdot)\in \mathcal{U}^\beta$ with given $\beta \ge 4$. Then for any $v(\cdot)=u(\cdot)-\bar{u}(\cdot)$ with $u(\cdot)\in\mathcal{U}^4$, it holds that 
\begin{equation}\label{initial-second-necessary-condition-thm}
  \begin{aligned}
\int_{\Gamma} \mathbb{E}\int_0^T \left\langle\mathbb{S}_\gamma(t)y_{1,\gamma}(t),  v(t)\right\rangle \mathrm{d}t \lambda^*(u; \mathrm{d}\gamma) \le 0,
\end{aligned}
\end{equation}
where $\mathbb{S}_\gamma(t)$ is given in~\eqref{defi-St} and $y_{1,\gamma}(t)$ in~\eqref{variant-equat-first}.
\end{thm}
The proof of Theorem \ref{thmfirornedes} is presented in  Section \ref{secprothfirinter}. 
To further obtain pointwise second-order necessary optimality conditions based on Theorem \ref{thmfirornedes}, firstly we derive the following variational inequalities with common reference uncertainty measure.
\begin{cor}\label{thmfirornedesf}
Under the conditions in Theorem \ref{thmfirornedes}, \textcolor{blue}{ together with \emph{(H5)},}  there exists $\lambda^*\in\Lambda^{\bar{u}(\cdot)}$, such that for any $v(\cdot)=u(\cdot)-\bar{u}(\cdot)$ with $u(\cdot)\in\mathcal{U}^4$, we have 
    \begin{equation}\label{yutythio}
      \begin{aligned}
    \int_{\Gamma} \mathbb{E}\int_0^T \left\langle\mathbb{S}_\gamma(t)y_{1,\gamma}(t),  v(t)\right\rangle \mathrm{d}t \lambda^*(\mathrm{d}\gamma) \le 0 . 
      \end{aligned}
      \end{equation}
\end{cor}
The proof of Corollary \ref{thmfirornedesf} is given in the later part of Section \ref{secprothfirinter}. 
The following theorem is the second major result of this article.
\begin{thm}\label{thm-point-nece-cla}
  Let Assumptions \emph{(H1)-(H6)} hold, and $\bar{u}(\cdot)\in \mathcal{U}^\beta$, $\beta \ge 4$ is a singular stochastic optimal control in the sense defined in Definition~\ref{singularoptimalcontroldef}. Then there exists $\lambda^*\in\Lambda^{\bar{u}(\cdot)}$, such that for $a.e.\  \tau\in[0,T]$ and $\forall v\in U$, 
\begin{align}\label{se-ne-thm-fi}
& \int_\Gamma\left\langle\mathbb{S}_\gamma(\tau)\partial_ub_\gamma(\tau)(v-\bar{u}(\tau)),v-\bar{u}(\tau)\right\rangle \lambda^*(\mathrm{d}\gamma)   \notag   \\
&\indent   +\int_\Gamma\left\langle\nabla\mathbb{S}_\gamma(\tau)\partial_u\sigma_\gamma(\tau)(v-\bar{u}(\tau)),v-\bar{u}(\tau)\right\rangle \lambda^*(\mathrm{d}\gamma)  \\
&\indent  -\int_\Gamma\left\langle \mathbb{S}_\gamma(\tau)\partial_u\sigma_\gamma(\tau)(v-\bar{u}(\tau)),\nabla\bar{u}(\tau)\right\rangle \lambda^*(\mathrm{d}\gamma)  \notag  \\
&\quad  \le0 \quad  a.s.  \notag 
\end{align}
\end{thm}
The proof of Theorem \ref{thm-point-nece-cla} is performed in Section \ref{seclapoisecpro}. 

\begin{rem}
    Note that there are cases satisfying Assumption $\mathrm{(H6)\ (i)}$, such as the linear quadratic optimal control problem with convex control constraints presented in \cite[Section 3.3]{zhanghaisenzhangxu2015siam}. See also Example 3.16 and Example 3.17 in \cite{zhanghaisenzhangxu2015siam}. \emph{(H6) (ii)} can be satisfied such as $\Gamma:=\mathbb{N}$ endowed with the metric inherited from the Euclidean.
    In general, any countable discrete space with the trivial topology $\mathsf{d} (\gamma,\gamma') = 1_{\gamma\not=\gamma'}$ setting satisfies \emph{(H6) (ii)}. 
\end{rem}

\begin{rem}
The Assumption \emph{(H6) (i)} about $\mathbb{L}_{2,\mathbb{F}}^{1,2}(\cdot)$ restrictions are mainly for guaranteeing the existence of certain Malliavin derivatives, while the $L^\infty$ restrictions are for Clark-Ocone theorem. Besides,  \emph{(H6) (ii)} is about the measurable of $\mathcal{D}_\cdot\mathbb{S}_\cdot(\cdot)$, with the help of the partitions of unity theorem.  
The constraints on Malliavin derivatives \emph{(H6)} are essentially the same as those on ordinary derivatives as \emph{(H3)} and \emph{(H4)} separately. 
We also have a complementary comment in appendix. 
\end{rem}

\begin{rem}
The contents in sections \ref{secprothfirinter} and \ref{seclapoisecpro} are inspired and provided on the basis of \cite{zhanghaisenzhangxu2015siam}, especially its utilization of tools in Malliavin calculus to deal with difficulties in the integrable order deficiency for multiple integrals of both Lebesgue and It\^o types, which are followed by lots of works studying the second-order necessary optimality conditions (cf. \cite{luqi2016conference,luqizhanghaisenzhangxu2021siam,zhanghaisenzhangxu2017siam,zhanghaisenzhangxu2018siamreview,zhanghaisenzhangxu2016scm}). 
It is natural to study a robust version of their works further. To achieve the purpose, technically, we need to establish the measurability and uniform regularity of the second-order derivatives of the value function and related items arising from it. Besides, we should perform weak convergence arguments within second-order nonlinear setting, 
and establish lemmas of Lebesgue differentiation and Malliavin approximation tailored to our model.
\end{rem}

\section{Motivating example}\label{secexam}
In this section, we apply the main results to an example to show the effectiveness of our results in  narrowing the candidate sets of optimal controls when the classical maximum principle can not provide useful information.
 
To present the example as straightforward as possible while it works,  let $m=n=1, T=1, U=[-1,1]$, $\Gamma=\{1,2\}$, $\Lambda=\{\lambda^ \mathfrak{i} | \mathfrak{i} \in[0,1]\}$, and $\lambda^\mathfrak{i} (\{1\})= \mathfrak{i} , \lambda^\mathfrak{i} (\{2\})=1- \mathfrak{i} $.
The control systems corresponding to uncertainty parameter $\gamma=1,2$ are respectively
\begin{equation*}
\left\{
\begin{aligned}
&\mathrm{d}x_1(t)= u(t)\mathrm{d}t+u(t)\mathrm{d}W(t), \quad t\in[0,1],         \\
&x_1(0)=0,
\end{aligned}
\right.
\quad  \text{and} \quad 
\left\{
\begin{aligned}
&\mathrm{d}x_2(t)=u(t)\mathrm{d}t, \quad t\in[0,1],           \\
&x_2(0)=0.
\end{aligned}
\right.
\end{equation*}
Take 
\begin{align*}
    f_1(t,x_1(t),u(t)) = \frac{1}{2}|u(t)|^2,   &\quad  
    f_2(t,x_2(t),u(t)) =  \frac{1}{4}|u(t)|^4 ,  
    \\ 
    h_1(x_1(1))=-\frac{1}{2}|x_1(1)|^2, &\quad h_2(x_2(1))=-\frac{1}{2}|x_2(1)|^2. 
\end{align*}
Then the cost functional is given by
\begin{align*}
  J(u(\cdot))&=\sup_{\lambda^ \mathfrak{i} \in\Lambda}\int_{\Gamma}\mathbb{E}\Big[\int_0^1 f_\gamma(t,x_\gamma(t),u(t))\mathrm{d}t
       +h_\gamma(x_\gamma(1))\Big]\lambda^\mathfrak{i} (\mathrm{d}\gamma) \\
       &=\sup_{ \mathfrak{i} \in[0,1]}\Big\{  \mathfrak{i} \Big(\frac{1}{2} \mathbb{E}\int_0^1|u(t)|^2\mathrm{d}t-\frac{1}{2}\mathbb{E}|x_1(1)|^2\Big)  +(1- \mathfrak{i} ) \Big(\frac{1}{4}\mathbb{E}\int_0^1|u(t)|^4\mathrm{d}t-\frac{1}{2}\mathbb{E}|x_2(1)|^2\Big) \Big\}.
\end{align*}

Correspondingly for $\gamma=1,2$, the adjoint equations are
\begin{equation*}
\left\{
\begin{aligned}
&\mathrm{d}P_{1,\gamma}(t)=Q_{1,\gamma}(t)\mathrm{d}W(t), \quad t\in[0,1],  \\
&P_{1,\gamma}(1)=0,
\end{aligned}
\right.
\quad 
\text{and}
\quad
\left\{
\begin{aligned}
&\mathrm{d}P_{2,\gamma}(t)=Q_{2,\gamma}(t)\mathrm{d}W(t), \quad t\in[0,1],  \\
&P_{2,\gamma}(1)=1.
\end{aligned}
\right.
\end{equation*}
Then
$$(P_{1,\gamma}(t),Q_{1,\gamma}(t))\equiv(0,0),\quad (P_{2,\gamma}(t),Q_{2,\gamma}(t))\equiv(1,0),\quad  t\in[0,1],\  \gamma=1,2.$$

Besides, the Hamiltonians  are separately 
\begin{equation*}\label{hamiltonianexa1}
\begin{aligned}
  H_{1}(t,x_1,u,p_{1,1},q_{1,1})&:=\langle p_{1,1},b_1(t,x_1,u)\rangle +\langle q_{1,1},\sigma_1(t,x_1,u)\rangle  -f_1(t,x_1,u)    =p_{1,1}u+q_{1,1}u-\frac{1}{2}u^2,
\end{aligned}
\end{equation*}
and
\begin{equation*}\label{hamiltonianexa2}
\begin{aligned}
  H_{2}(t,x_2,u,p_{1,2},q_{1,2})&:=\langle p_{1,2},b_2(t,x_2,u)\rangle +\langle q_{1,2},\sigma_2(t,x_2,u)\rangle  -f_2(t,x_2,u)    =p_{1,2}u -\frac{1}{4}u^4,
\end{aligned}
\end{equation*}
where $(t,x_\gamma,u,p_{1,\gamma},q_{1,\gamma})\in[0,1]\times\mathbb{R}\times [-1,1]\times\mathbb{R}\times\mathbb{R}$.

Choose $\bar{u}(\cdot) \equiv 0$ on $t\in[0,1]$. Then the corresponding states are $\bar{x}_1(\cdot)\equiv0, \bar{x}_2(\cdot)\equiv0$, and $J(\bar{u}(\cdot))=0$.

Then for $\gamma=1$ and a.e. $(t,\omega)\in[0,1]\times \Omega$,
\begin{align*}
&\partial_u H_{1}(t,\bar{x}_1(t),\bar{u}(t),P_{1,1}(t),Q_{1,1}(t))    =P_{1,1}(t)+Q_{1,1}(t)-\bar{u}(t)=0,   \\
&\partial_{uu} H_{1}(t,\bar{x}_1(t),\bar{u}(t),P_{1,1}(t),Q_{1,1}(t)) + \partial_u \sigma_1(t,\bar{x}_1(t),\bar{u}(t))^\top P_{2,1}(t) \partial_u \sigma_1(t,\bar{x}_1(t),\bar{u}(t))    =-1+1=0.
\end{align*}
Besides, for $\gamma=2$ and a.e. $(t,\omega)\in[0,1]\times \Omega$,
\begin{align*}
&\partial_u H_{2}(t,\bar{x}_2(t),\bar{u}(t),P_{1,2}(t),Q_{1,2}(t))    =P_{1,2}(t)-\bar{u}^3(t)=0,   \\
& \partial_{uu} H_{2}(t,\bar{x}_2(t),\bar{u}(t),P_{1,2}(t),Q_{1,2}(t)) + \partial_u \sigma_2(t,\bar{x}_2(t),\bar{u}(t))^\top P_{2,2}(t) \partial_u \sigma_2(t,\bar{x}_2(t),\bar{u}(t))    =0.
\end{align*}
Then
\begin{equation*}
\left\{
\begin{aligned}
&\int_\Gamma\partial_u H_\gamma(t,\bar{x}_\gamma(t),\bar{u}(t),P_{1,\gamma}(t),Q_{1,\gamma}(t))\lambda^\mathfrak{i} (\mathrm{d}\gamma)=0 \quad    a.s.\  a.e.\   t\in[0,1],\  \mathfrak{i} \in[0,1],
\\&\int_\Gamma\Big[\partial_{uu}H_{\gamma}(t,\bar{x}_\gamma(t),\bar{u}(t),P_{1,\gamma}(t),Q_{1,\gamma}(t))  +\partial_u \sigma_\gamma(t,\bar{x}_\gamma(t),\bar{u}(t))^\top \\
&\indent\indent\indent\indent  P_{2,\gamma}(t) \partial_u \sigma_\gamma(t,\bar{x}_\gamma(t),\bar{u}(t))\Big]\lambda^\mathfrak{i} (\mathrm{d}\gamma) =0  \quad    a.s.\  a.e.\   t\in[0,1],\  \mathfrak{i} \in[0,1].
\end{aligned}
\right.
\end{equation*}
That is to say, $\bar{u}(t)\equiv0, t\in[0,1]$ is a singular control in the classical sense in Definition~\ref{singularoptimalcontroldef}, and the classical maximum principle (first-order necessary conditions for optimality) is trivial in this case.

However, by taking $u(t)\equiv1, t\in[0,1]$,
\begin{align*}
  J(u(\cdot))& =\sup_{ \mathfrak{i} \in[0,1]}\Big\{  \mathfrak{i} \Big(\frac{1}{2}\mathbb{E}\int_0^1|u(t)|^2\mathrm{d}t-\frac{1}{2}\mathbb{E}|x_1(1)|^2\Big)   +(1- \mathfrak{i} ) \Big(\frac{1}{4}\mathbb{E}\int_0^1|u(t)|^4\mathrm{d}t-\frac{1}{2}\mathbb{E}|x_2(1)|^2\Big) \Big\}    \\
  & =\sup_{ \mathfrak{i} \in[0,1]}\Big\{ (-\frac{1}{2}) \mathfrak{i}  + (-\frac{1}{4})(1- \mathfrak{i} ) \Big\}   \le-\frac{1}{4}< 0 =J(\bar{u}(\cdot)).
\end{align*}
This implies that $\bar{u}(t)\equiv0, t\in[0,1]$ is not one of the optimal control.

Actually, the fact that $\bar{u}(t)\equiv0, t\in[0,1]$ is not one of the optimal control can be obtained by applying Theorem~\ref{thm-point-nece-cla},   showing that $\bar{u}(t)\equiv0, t\in[0,1]$ does not satisfy~\eqref{se-ne-thm-fi}.

Indeed, it can be verified that for any $t\in[0,1]$,
\begin{align*}
& \mathbb{S}_1(t)=\mathbb{S}_1(t,x_1(t),\bar{u}(t),P_{1,1}(t),Q_{1,1}(t),P_{2,1}(t),Q_{2,1}(t))\equiv1,  \\
& \mathbb{S}_2(t)=\mathbb{S}_2(t,x_2(t),\bar{u}(t),P_{1,2}(t),Q_{1,2}(t),P_{2,2}(t),Q_{2,2}(t))\equiv1,  \\
& \nabla\mathbb{S}_1(t)\equiv0,\quad  \nabla\mathbb{S}_2(t)\equiv0,\quad  \nabla\bar{u}(t)\equiv0.
\end{align*}
Take $v=1$ for example, then for almost all $(\tau,\omega)\in[0,1]\times \Omega$,
\begin{align*}
&\int_{\Gamma} \Big\langle\mathbb{S}_\gamma(\tau) \partial_u b_\gamma(\tau) (v-\bar{u}(\tau)),v-\bar{u}(\tau)\Big\rangle \lambda^*(\mathrm{d}\gamma)  \\
&\indent  + \int_{\Gamma} \Big\langle  \nabla\mathbb{S}_\gamma(\tau)  \partial_u\sigma_\gamma(\tau) (v-\bar{u}(\tau)), v-\bar{u}(\tau)\Big\rangle \lambda^*(\mathrm{d}\gamma)  \\
&\indent -\int_{\Gamma} \Big\langle  \mathbb{S}_\gamma(\tau) \partial_u\sigma_\gamma(\tau) (v-\bar{u}(\tau)),  \nabla\bar{u}(\tau)\Big\rangle \lambda^*(\mathrm{d}\gamma)  \\
&\quad =1,
\end{align*}
which is inconsistent with~\eqref{se-ne-thm-fi}.

\section{The proof of Theorem~\ref{thmfirornedes}}\label{secprothfirinter}
\begin{proof}[The proof of Theorem~\ref{thmfirornedes}]
The proofs are divided into two steps.
Step 1 primarily incorporates variational analysis and dual principles, while Step 2 mainly involves weak convergence analysis as utilized in \cite{hlw23,humingshangwangfalei2020siam} generalized to this singular setting. Both steps rely significantly on the regularities with respect to the uncertainty parameters estimated in Section \ref{seformmainres}.

\textbf{Step 1:}  \
In this step we prove that
\begin{equation}\label{delt-j-low-limit-eps}
  \begin{aligned}
  &\mathop{\underline{\lim}}_{\varepsilon\to 0} \frac{1}{\varepsilon^2} \left[J(u^\varepsilon(\cdot))-J(\bar{u}(\cdot))\right]  \geq\sup_{\lambda\in\Lambda^{\bar{u}}}\Big\{-\int_{\Gamma} \mathbb{E}\int_0^T \left\langle\mathbb{S}_\gamma(t)y_{1,\gamma}(t),v(t)\right\rangle \mathrm{d}t  \lambda(\mathrm{d}\gamma)\Big\}.
  \end{aligned}
\end{equation}

Take $\lambda\in\Lambda^{\bar{u}}$,
then
$J(\bar{u}(\cdot))=\int_{\Gamma}\mathbb{E} \big[\int_0^Tf_\gamma(t,\bar{x}_\gamma(t),\bar{u}(t))\mathrm{d}t
       +h_\gamma(\bar{x}_\gamma(T))\big]\lambda(\mathrm{d}\gamma)$
and
$J(u^\varepsilon(\cdot))\geq\int_{\Gamma}\mathbb{E} \big[\int_0^Tf_\gamma(t,x_\gamma^\varepsilon(t),u^\varepsilon(t))\mathrm{d}t
       +h_\gamma(x_\gamma^\varepsilon(T))\big]\lambda(\mathrm{d}\gamma).$
Then
\begin{equation}\label{J-epsi-J-bar-left--1}
\begin{aligned}
&J(u^\varepsilon(\cdot))-J(\bar{u}(\cdot)) \geq \int_{\Gamma}\Big\{\mathbb{E} \int_0^T \left[f_\gamma(t,x_\gamma^\varepsilon(t),u^\varepsilon(t)) -f_\gamma(t,\bar{x}_\gamma(t),\bar{u}(t))\right]\mathrm{d}t \\
&\indent\indent\indent\indent\indent\indent\indent +\mathbb{E}\left[h_\gamma(x_\gamma^\varepsilon(T)) -h_\gamma(\bar{x}_\gamma(T))\right]\Big\}\lambda(\mathrm{d}\gamma).
\end{aligned}
\end{equation}
Moreover,  by~\eqref{est-sup-theta-f-h-ep-six1} and~\eqref{est-sup-theta-f-h-ep-six2}, we get the following Taylor expansion of the difference of the cost functional
  \begin{align} \label{J-epsi-J-bar-left--2}
& \int_{\Gamma}\Big\{\mathbb{E} \int_0^T \left[f_\gamma(t,x_\gamma^\varepsilon(t),u^\varepsilon(t)) -f_\gamma(t,\bar{x}_\gamma(t),\bar{u}(t))\right]\mathrm{d}t   +\mathbb{E}\left[h_\gamma(x_\gamma^\varepsilon(T)) -h_\gamma(\bar{x}_\gamma(T))\right]\Big\}\lambda(\mathrm{d}\gamma)  \notag  \\
&\indent  =\int_{\Gamma}\Big\{\mathbb{E}\int_0^T\Big[\langle\partial_x f_{\gamma}(t),\delta x_{\gamma}(t)\rangle +\varepsilon \langle\partial_u f_{\gamma}(t),v(t)\rangle +\varepsilon\langle\partial_{xu}f_{\gamma}(t)\delta x_{\gamma}(t),v(t)\rangle \notag \\
&\indent\indent\indent\indent \ +\frac{1}{2}\langle\partial_{xx}f_\gamma(t)\delta x_\gamma(t),\delta x_{\gamma}(t)\rangle
+\frac{{\varepsilon}^2}{2}\langle\partial_{uu}f_{\gamma}(t)v(t),v(t)\rangle\Big]\mathrm{d}t \notag \\
&\quad\indent\quad +\mathbb{E}\Big[\langle\partial_x h_\gamma(\bar{x}_\gamma(T)), \delta x_\gamma(T)\rangle +\frac{1}{2}\langle\partial_{xx}h_\gamma(\bar{x}_\gamma(T))\delta x_\gamma(T),\delta x_\gamma(T)\rangle\Big]\Big\}\lambda(\mathrm{d}\gamma)  +\int_{\Gamma} o(\varepsilon_\gamma^2) \lambda(\mathrm{d}\gamma)  \notag 
\\  
&\indent =\int_{\Gamma}\Big\{\mathbb{E}\int_0^T\Big[\varepsilon\langle\partial_x f_{\gamma}(t),y_{1,\gamma}(t)\rangle +\frac{\varepsilon^2}{2} \langle\partial_x f_{\gamma}(t),y_{2,\gamma}(t)\rangle +\varepsilon\langle\partial_{u}f_{\gamma}(t),v(t)\rangle \notag \\
&\indent\indent\indent\indent \ +\frac{\varepsilon^2}{2} \big( \langle\partial_{xx}f_\gamma(t)y_{1,\gamma}(t),y_{1,\gamma}(t)\rangle  +2 \langle\partial_{xu}f_\gamma(t)y_{1,\gamma}(t),v(t)\rangle \notag  \\
&\indent\indent\indent\indent \ +\langle\partial_{uu}f_{\gamma}(t)v(t),v(t)\rangle\big)\Big]\mathrm{d}t  \\
&\quad\indent\indent +\mathbb{E}\Big[\varepsilon\langle\partial_x h_\gamma(\bar{x}_\gamma(T)), y_{1,\gamma}(T)\rangle +\frac{\varepsilon^2}{2}\langle\partial_{xx}h_\gamma(\bar{x}_\gamma(T))y_{1,\gamma}(T), y_{1,\gamma}(T)\rangle \notag  \\
&\indent\indent\indent\  +\frac{\varepsilon^2}{2}\langle\partial_x h_\gamma(\bar{x}_\gamma(T)), y_{2,\gamma}(T)\rangle\Big]\Big\}\lambda(\mathrm{d}\gamma) +o(\varepsilon^2),   \quad    \varepsilon\to0. \notag
  \end{align}
The last equality makes use of   Proposition~\ref{lem-ele-y1y2-delta-x-esti} to approximate $\delta x_\gamma$ by $y_{1,\gamma}$ and $y_{2,\gamma}$, and to transform 
$\int_{\Gamma} o(\varepsilon_\gamma^2) \lambda(\mathrm{d}\gamma)$
into $\int_{\Gamma} o(\varepsilon^2) \lambda(\mathrm{d}\gamma)$.

Now we utilize the duality arguments with the help of the two adjoint equations introduced in \eqref{adjoint-equa-first} and \eqref{adjoint-equa-second}  to get rid of  terms $\langle y_{1,\gamma},y_{1,\gamma}\rangle $ and $y_{2,\gamma}$. It can be verified directly by It\^{o}'s formula that for any $\gamma\in\Gamma$,
\begin{equation}\label{equa-ito-formula-firs}
\begin{aligned}
&  \mathbb{E}\langle \partial_x h_\gamma(\bar{x}_\gamma(T)),y_{1,\gamma}(T) \rangle
= -\mathbb{E}\langle P_{1,\gamma}(T),y_{1,\gamma}(T)\rangle    \\
& \indent
=  -\mathbb{E}\int_0^T \big[ \langle P_{1,\gamma}(t),\partial_u b_\gamma(t)v(t)\rangle +\langle Q_{1,\gamma}(t),\partial_u\sigma_\gamma(t)v(t)\rangle +\langle\partial_x f_\gamma(t),y_{1,\gamma}(t)\rangle \big]\mathrm{d}t,
\end{aligned}
\end{equation}
\begin{align}\label{equa-ito-formula-second}
&  \mathbb{E}\langle \partial_x h_\gamma(\bar{x}_\gamma(T)),y_{2,\gamma}(T) \rangle  = -\mathbb{E}\langle P_{1,\gamma}(T),y_{2,\gamma}(T)\rangle  \notag  \\
& \indent = -\mathbb{E}\int_0^T \Big[\langle P_{1,\gamma}(t),y_{1,\gamma}(t)^\top\partial_{xx}b_\gamma(t)y_{1,\gamma}(t)\rangle +2\langle P_{1,\gamma}(t),v(t)^\top\partial_{xu}b_\gamma(t)y_{1,\gamma}(t)\rangle   \notag \\
& \indent\indent\indent +\langle P_{1,\gamma}(t),v(t)^\top\partial_{uu}b_\gamma(t)v(t)\rangle  +\langle Q_{1,\gamma}(t),y_{1,\gamma}(t)^\top\partial_{xx}\sigma_\gamma(t)y_{1,\gamma}(t)\rangle    \\
& \indent\indent\indent +2\langle Q_{1,\gamma}(t),v(t)^\top\partial_{xu}\sigma_\gamma(t)y_{1,\gamma}(t)\rangle
+\langle Q_{1,\gamma}(t),v(t)^\top\partial_{uu}\sigma_\gamma(t)v(t)\rangle  +\langle\partial_x f_\gamma(t),y_{2,\gamma}(t)\rangle\Big]\mathrm{d}t,  \notag 
\end{align}
and
\begin{align}\label{equa-ito-formula-third}
&  \mathbb{E}\langle \partial_{xx} h_\gamma(\bar{x}_\gamma(T))y_{1,\gamma}(T),y_{1,\gamma}(T)\rangle
 = -\mathbb{E}\langle P_{2,\gamma}(T)y_{1,\gamma}(T),y_{1,\gamma}(T)\rangle  \notag  \\
& \indent =-\mathbb{E}\int_0^T \Big[ \langle P_{2,\gamma}(t)y_{1,\gamma}(t),\partial_u b_\gamma(t)v(t)\rangle +\langle P_{2,\gamma}(t)\partial_u b_\gamma(t)v(t),y_{1,\gamma}(t)\rangle  \notag  \\
& \indent\indent\indent\indent   +\langle P_{2,\gamma}(t)\partial_x\sigma_\gamma(t)y_{1,\gamma}(t),\partial_u\sigma_\gamma(t)v(t)\rangle
+\langle P_{2,\gamma}(t)\partial_u\sigma_\gamma(t)v(t),\partial_x\sigma_\gamma(t)y_{1,\gamma}(t)\rangle \notag \\
& \indent\indent\indent\indent  +\langle Q_{2,\gamma}(t)\partial_u\sigma_\gamma(t)v(t), y_{1,\gamma}(t)\rangle
+\langle Q_{2,\gamma}(t)y_{1,\gamma}(t),\partial_u\sigma_\gamma(t)v(t)\rangle \notag \\
&\indent\indent\indent\indent  +\langle P_{2,\gamma}(t)\partial_u\sigma_\gamma(t)v(t),\partial_u\sigma_\gamma(t)v(t)\rangle
-\langle\partial_{xx}H_\gamma(t) y_{1,\gamma}(t),y_{1,\gamma}(t)\rangle \Big]\mathrm{d}t  \\
& \indent =-\mathbb{E}\int_0^T \Big[ 2\langle P_{2,\gamma}(t)y_{1,\gamma}(t),\partial_ub_\gamma(t)v(t)\rangle +2\langle Q_{2,\gamma}(t)\partial_u\sigma_\gamma(t)v(t), y_{1,\gamma}(t)\rangle  \notag \\
&\indent\indent\indent\indent  +\langle P_{2,\gamma}(t)\partial_u\sigma_\gamma(t)v(t),\partial_u\sigma_\gamma(t)v(t)\rangle
-\langle\partial_{xx}H_\gamma(t) y_{1,\gamma}(t),y_{1,\gamma}(t)\rangle  \notag\\
& \indent\indent\indent\indent   +2\langle P_{2,\gamma}(t)\partial_x\sigma_\gamma(t)y_{1,\gamma}(t),\partial_u\sigma_\gamma(t)v(t)\rangle \Big]\mathrm{d}t. \notag
\end{align}

Then by the estimates \eqref{J-epsi-J-bar-left--2}-\eqref{equa-ito-formula-third}, it derives 
\begin{align}\label{J-epsi-J-bar-left--3}
  & \int_{\Gamma}\Big\{\mathbb{E} \int_0^T \left[f_\gamma(t,x_\gamma^\varepsilon(t),u^\varepsilon(t)) -f_\gamma(t,\bar{x}_\gamma(t),\bar{u}(t))\right]\mathrm{d}t  +\mathbb{E}\left[h_\gamma(x_\gamma^\varepsilon(T)) -h_\gamma(\bar{x}_\gamma(T))\right]\Big\}\lambda(\mathrm{d}\gamma)\notag \\
  &\indent  =-\int_{\Gamma}\Big\{\mathbb{E}\int_0^T\Big[ \varepsilon\big(\left\langle P_{1,\gamma}(t),\partial_u b_\gamma(t)v(t)\right\rangle +\left\langle Q_{1,\gamma}(t),\partial_u \sigma_\gamma(t)v(t)\right\rangle  -\langle\partial_u f_\gamma(t)v(t)\rangle\big)  \notag  \\
  &\indent\indent\indent\indent\indent   +\frac{\varepsilon^2}{2}\Big(\left\langle P_{1,\gamma}(t),v(t)^\top\partial_{uu} b_\gamma(t)v(t)\right\rangle+\left\langle Q_{1,\gamma}(t),v(t)^\top\partial_{uu} \sigma_\gamma(t)v(t)\right\rangle  \notag  \\
  &\indent\indent\indent\indent\indent\indent +\left\langle P_{2,\gamma}(t)\partial_u\sigma_\gamma(t)v(t),\partial_u\sigma_\gamma(t)v(t)\right\rangle -\langle\partial_{uu}f_\gamma(t)v(t),v(t)\rangle\Big) \notag 
\\  
&\indent\indent\indent\indent\indent  +\varepsilon^2\Big(\left\langle P_{1,\gamma}(t),v(t)^\top\partial_{xu} b_\gamma(t)y_{1,\gamma}(t)\right\rangle +\left\langle Q_{1,\gamma}(t),v(t)^\top\partial_{xu} \sigma_\gamma(t)y_{1,\gamma}(t)\right\rangle  \notag  \\&\indent\indent\indent\indent\indent\indent
-\left\langle\partial_{xu}f_\gamma(t)y_{1,\gamma}(t),v(t)\right\rangle +\left\langle\partial_u b_\gamma(t)^\top P_{2,\gamma}(t)y_{1,\gamma}(t),v(t)\right\rangle  \notag  \\&
\indent\indent\indent\indent\indent\indent   +\left\langle\partial_u\sigma_\gamma(t)^\top P_{2,\gamma}(t)\partial_x\sigma_\gamma(t)y_{1,\gamma}(t),v(t)\right\rangle  \notag \\&
\indent\indent\indent\indent\indent\indent   +\left\langle\partial_u\sigma_\gamma(t)^\top Q_{2,\gamma}(t)y_{1,\gamma}(t),v(t)\right\rangle\Big)
\Big]\mathrm{d}t\Big\}\lambda(\mathrm{d}\gamma) +o(\varepsilon^2) \notag \\
&\indent =-\int_{\Gamma}\Big\{\mathbb{E}\int_0^T\Big[\varepsilon\langle\partial_u H_\gamma(t),v(t)\rangle
+\varepsilon^2\left\langle\mathbb{S}_\gamma(t)y_{1,\gamma}(t),v(t)\right\rangle \notag  \\
& \indent\indent\indent\indent
+\frac{\varepsilon^2}{2}\left\langle\left(\partial_{uu}H_\gamma(t)+\partial_u\sigma_\gamma(t)^\top P_{2,\gamma}(t) \partial_u\sigma_\gamma(t)\right)v(t),v(t)\right\rangle  \Big]\mathrm{d}t\Big\}\lambda(\mathrm{d}\gamma)  \\
&\indent\indent  +o(\varepsilon^2), \quad    \varepsilon\to 0.  \notag
\end{align}

Recalling the definition of singular optimal control in Definition~\ref{singularoptimalcontroldef},  by~\eqref{J-epsi-J-bar-left--3}, we obtain 
\begin{align}\label{J-epsi-J-bar-left--4}
&\lim_{\varepsilon\to 0} \frac{1}{\varepsilon^2} \int_{\Gamma}\Big\{\mathbb{E} \int_0^T \left[f_\gamma(t,x_\gamma^\varepsilon(t),u^\varepsilon(t)) -f_\gamma(t,\bar{x}_\gamma(t),\bar{u}(t))\right]\mathrm{d}t +\mathbb{E}\left[h_\gamma(x_\gamma^\varepsilon(T)) -h_\gamma(\bar{x}_\gamma(T))\right]\Big\}\lambda(\mathrm{d}\gamma)   \\
&\indent =  -\int_{\Gamma} \mathbb{E}\int_0^T
 \left\langle\mathbb{S}_\gamma(t)y_{1,\gamma}(t),v(t)\right\rangle \mathrm{d}t \lambda(\mathrm{d}\gamma). \notag 
\end{align}
Since~\eqref{J-epsi-J-bar-left--1} and~\eqref{J-epsi-J-bar-left--4}
holds for any $\lambda\in\Lambda^{\bar{u}}$,  \eqref{delt-j-low-limit-eps} holds.

\textbf{Step 2:}  \
In this step, we prove the existence of $\lambda^*\in\Lambda^{\bar{u}}$ satisfying 
\begin{equation}\label{delt-j-up-limit}
  \begin{aligned}
  &   \mathop{\overline{\lim}}_{\varepsilon\to 0}\frac{1}{\varepsilon^2} \left[J(u^\varepsilon(\cdot))-J(\bar{u}(\cdot))\right]    \le -\int_{\Gamma}\big[\mathbb{E}\int_0^T \left\langle\mathbb{S}_\gamma(t)y_{1,\gamma}(t),v(t)\right\rangle \mathrm{d}t    \big]\lambda^*(\mathrm{d}\gamma).
  \end{aligned}
  \end{equation}

In the singular optimal control case, we can choose a sub-sequence $\varepsilon_N\rightarrow 0$, such that
\begin{equation*}
  \mathop{\overline{\lim}}_{\varepsilon\to 0}\frac{1}{\varepsilon^2} [J(u^\varepsilon(\cdot))-J(\bar{u}(\cdot))]= \lim_{N\to \infty} \frac{1}{\varepsilon_N^2} [J(u^{\varepsilon_N}(\cdot))-J(\bar{u}(\cdot))].
\end{equation*}
Moreover, for each $N\ge1$, since $\Lambda^{u^{\varepsilon_N}}$ is nonempty,  we can choose $\lambda^{\varepsilon_N}\in\Lambda^{u^{\varepsilon_N}}$ such that
\[J(u^{\varepsilon_N}(\cdot))=\int_{\Gamma}\mathbb{E} \Big[\int_0^T f_\gamma(t,x_\gamma^{\varepsilon_N}(t),u^{\varepsilon_N}(t))\mathrm{d}t
       +h_\gamma(x_\gamma^{\varepsilon_N}(T))\Big]\lambda^{\varepsilon_N}(\mathrm{d}\gamma).\]
Meanwhile, by the definition of $J(\bar{u}(\cdot))$, it deduces that 
\[J(\bar{u}(\cdot))\geq\int_{\Gamma}\mathbb{E} \Big[\int_0^T f_\gamma(t,\bar{x}_\gamma(t),\bar{u}(t))\mathrm{d}t
       +h_\gamma(\bar{x}_\gamma(T))\Big]\lambda^{\varepsilon_N}(\mathrm{d}\gamma).\]
Together with~\eqref{wpec} and \eqref{lem-cost-continuous-theta-theta-pron-descri}, under Assumption (H2),    
there exists a $\lambda^*\in\Lambda$ such that ${(\lambda^{\varepsilon_N})}_{N\geq1}$ converges weakly to $\lambda^*$ (choosing a sub-sequence if necessary but still denoting as  ${(\lambda^{\varepsilon_N})}_{N\geq1}$ for simplicity).  
Moreover, we claim that $\lambda^*\in\Lambda^{\bar{u}}$. 
Actually,
  \begin{align*}
\lim_{N\to\infty} \left|J(u^{\varepsilon_N}(\cdot))-J(\bar{u}(\cdot))\right| 
&
\leq \lim_{N\to\infty}\Big| \int_{\Gamma} \Big\{\mathbb{E} \big[\int_0^T f_\gamma(t,x_\gamma^{\varepsilon_N}(t),u^{\varepsilon_N}(t))\mathrm{d}t
       +h_\gamma(x_\gamma^{\varepsilon_N}(T)) \big]   \\
 &\indent\indent\indent\indent      -\mathbb{E} \big[\int_0^T f_\gamma(t,\bar{x}_\gamma(t),\bar{u}(t))\mathrm{d}t
       +h_\gamma(\bar{x}_\gamma(T))\big] \Big\}\lambda^{\varepsilon_N}(\mathrm{d}\gamma)  \Big|  \\
&  \le \lim_{N\to\infty} \sup_{\gamma\in\Gamma} \Big|\mathbb{E} \big[\int_0^T f_\gamma(t,x_\gamma^{\varepsilon_N}(t),u^{\varepsilon_N}(t))\mathrm{d}t
       +h_\gamma(x_\gamma^{\varepsilon_N}(T)) \big]   \\
 &\indent\indent\indent\indent - \mathbb{E} \big[\int_0^T f_\gamma(t,\bar{x}_\gamma(t),\bar{u}(t))\mathrm{d}t
       +h_\gamma(\bar{x}_\gamma(T))\big] \Big|
 \\&    =0,
  \end{align*}
which further implies that
  \begin{align*}
 J(\bar{u}(\cdot))&= \lim_{N\to\infty} J(u^{\varepsilon_N}(\cdot))   =
\lim_{N\to\infty}  \int_{\Gamma}\mathbb{E} \big[\int_0^T f_\gamma(t,x_\gamma^{\varepsilon_N}(t),u^{\varepsilon_N}(t))\mathrm{d}t
       +h_\gamma(x_\gamma^{\varepsilon_N}(T))\big]\lambda^{\varepsilon_N}(\mathrm{d}\gamma)  \\
&  =\lim_{N\to\infty}  \int_{\Gamma}  \mathbb{E} \big[\int_0^T f_\gamma(t,\bar{x}_\gamma(t),\bar{u}(t))\mathrm{d}t
       +h_\gamma(\bar{x}_\gamma(T))\big]\lambda^{\varepsilon_N}(\mathrm{d}\gamma)    \\
&  =  \int_{\Gamma}  \mathbb{E} \big[\int_0^Tf_\gamma(t,\bar{x}_\gamma(t),\bar{u}(t))\mathrm{d}t
       +h_\gamma(\bar{x}_\gamma(T))\big]\lambda^*(\mathrm{d}\gamma),
  \end{align*}
indicating that $\lambda^*\in\Lambda^{\bar{u}}$.
On the other hand,
\begin{align*}
J(u^{\varepsilon_N}(\cdot))-J(\bar{u}(\cdot))  
&
\leq  \int_{\Gamma} \Big\{\mathbb{E} \big[\int_0^Tf_\gamma(t,x_\gamma^{\varepsilon_N}(t),u^{\varepsilon_N}(t))\mathrm{d}t
       +h_\gamma(x_\gamma^{\varepsilon_N}(T)) \big]   \\
 &\indent\indent    -
       \mathbb{E} \big[\int_0^T f_\gamma(t,\bar{x}_\gamma(t),\bar{u}(t))\mathrm{d}t
       +h_\gamma(\bar{x}_\gamma(T))\big] \Big\}\lambda^{\varepsilon_N}(\mathrm{d}\gamma) 
       \\  &   
       =:I_1, 
\end{align*}
while  
\begin{equation}\label{exftel10}
\begin{aligned}
&  
     I_1 = -\int_{\Gamma} \mathbb{E}\int_0^T\Big\{
    \varepsilon_N\langle\partial_u H_\gamma(t),v(t)\rangle
+\varepsilon_N^2\left\langle\mathbb{S}_\gamma(t)y_{1,\gamma}(t),v(t)\right\rangle   \\
& \indent\indent\indent\indent\indent
+\frac{\varepsilon_N^2}{2}\left\langle\left(\partial_{uu}H_\gamma(t)+\partial_u\sigma_\gamma(t)^\top P_{2,\gamma}(t) \partial_u\sigma_\gamma(t)\right)v(t),v(t)\right\rangle  \Big\}\mathrm{d}t \lambda^{\varepsilon_N}(\mathrm{d}\gamma) \\
&\indent\indent\indent +o(\varepsilon_N^2), \quad  \varepsilon_N\to 0^+.
\end{aligned}
\end{equation}
By the definition of the singular optimal control and recall the weak convergence of the probability measure sequence ${\{\lambda^{\varepsilon_N}\}}_{N\ge1}$, it derives that 
\begin{equation}\label{ewuyrtuewio}
\begin{aligned}
\lim_{N\to\infty} \int_{\Gamma} \mathbb{E}\int_0^T
   \left\langle\left(\partial_{uu}H_\gamma(t)+\partial_u\sigma_\gamma(t)^\top P_{2,\gamma}(t) \partial_u\sigma_\gamma(t)\right)v(t),v(t)\right\rangle \mathrm{d}t   \lambda^{\varepsilon_N}(\mathrm{d}\gamma)=0.
\end{aligned}
\end{equation}
Besides, we claim that 
\begin{equation}\label{tec-epsi-firs-delt-j-limi}
\begin{aligned}
 \lim_{N\to\infty}  \frac{1}{\varepsilon_N}  \int_{\Gamma} \mathbb{E}\int_0^T \langle\partial_u H_\gamma(t),v(t)\rangle \mathrm{d}t \lambda^{\varepsilon_N}(\mathrm{d}\gamma)=0.
\end{aligned}
\end{equation}
Indeed, since $\lambda^*\in \Lambda^{\bar{u}}$ as just proved, we have 
   $\int_{\Gamma} \mathbb{E}\int_0^T {\langle\partial_u H_\gamma(t),v(t)\rangle} \mathrm{d}t \lambda^*(\mathrm{d}\gamma)=0$
from the definition of singular optimal control, 
and then obviously 
 $\lim_{N\to\infty}  \frac{1}{\varepsilon_N}  \int_{\Gamma} \mathbb{E}\int_0^T \langle\partial_u H_\gamma(t),v(t)\rangle \mathrm{d}t \lambda^*(\mathrm{d}\gamma)=0$.
As a result, to prove \eqref{tec-epsi-firs-delt-j-limi}, it is sufficient to prove 
     $\lim_{N\to\infty}  \frac{1}{\varepsilon_N}  \int_{\Gamma} \mathbb{E}\int_0^T \langle\partial_u H_\gamma(t),v(t)\rangle \mathrm{d}t (\lambda^{\varepsilon_N} - \lambda^*)(\mathrm{d}\gamma)=0$.
Actually, notice the facts that 
\begin{align*}
    &  \frac{1}{\varepsilon_N}  \int_{\Gamma}\Big\{\mathbb{E}\int_0^T \langle\partial_u H_\gamma(t),v(t)\rangle \mathrm{d}t \Big\}\left(\lambda^{\varepsilon_{N}}-\lambda^*\right)(\mathrm{d}\gamma)    \le  \frac{\left(\lambda^{\varepsilon_{N}}-\lambda^*\right)(\Gamma)}{\varepsilon_N}   \sup_{\gamma\in\Gamma}\Big\{\mathbb{E}\int_0^T \left\langle\partial_u H_\gamma(t),v(t)\right\rangle \mathrm{d}t \Big\},
\end{align*}
where ${\left(\lambda^{\varepsilon_{N}}-\lambda^*\right)(\Gamma)}/{\varepsilon_N}\equiv0$, and, as the proof for \eqref{essthete},  $\sup_{\gamma\in\Gamma} |\mathbb{E}\int_0^T \left\langle\partial_u H_\gamma(t),v(t)\right\rangle \mathrm{d}t |<\infty$ for $\beta \ge 4$.
Then~\eqref{tec-epsi-firs-delt-j-limi} is obtained.

Now \eqref{delt-j-up-limit} follows from~\eqref{exftel10}-\eqref{tec-epsi-firs-delt-j-limi}.

Note that the measure $\lambda^*(\mathrm{d}\gamma)$ is dependent on the given admissible control $u$. 
From~\eqref{delt-j-low-limit-eps}  together with~\eqref{delt-j-up-limit}, for any $u(\cdot)\in\mathcal{U}^4$,  there exists $\lambda^*(u;\mathrm{d}\gamma)\in \Lambda^{\bar{u}}$, such that
\begin{equation*}
\begin{aligned}
\mathop{\lim}_{\varepsilon\to 0}\frac{1}{\varepsilon^2} \left[J(u^\varepsilon(\cdot))-J(\bar{u}(\cdot))\right]  =  -\int_{\Gamma}\mathbb{E}\int_0^T \left\langle\mathbb{S}_\gamma(t)y_{1,\gamma}(t),v(t)\right\rangle \mathrm{d}t \lambda^*(u;\mathrm{d}\gamma).
\end{aligned}
\end{equation*}
By the optimality of $\bar{u}$, and recalling the notations in \eqref{notaip}, it derives 
\begin{equation}\label{dfjb}
  \begin{aligned}
    \mathfrak{I} (v;\lambda^*(u;\mathrm{d}\gamma)) & = -\int_{\Gamma} \mathbb{E}\int_0^T \left\langle\mathbb{S}_\gamma(t)y_{1,\gamma}(t), v(t)\right\rangle \mathrm{d}t \lambda^*(u;\mathrm{d}\gamma) 
    \ge 0. 
\end{aligned}
\end{equation}
The proof is completed. 
\end{proof}
Based on Theorem \ref{thmfirornedes}, now we are ready to give 
\begin{proof}[{The proof of Corollary \ref{thmfirornedesf}}]
The main content is to derive the variational inequality with a common reference probability for any admissible control. 
The inequality \eqref{dfjb} yields
\begin{equation}\label{isdg}
    \begin{aligned}
\inf_{u(\cdot)\in\mathcal{U}^4}   \sup_{\lambda^*(\mathrm{d}\gamma)\in \Lambda^{\bar{u}}} \mathfrak{I} (v;\lambda^*(\mathrm{d}\gamma)) \ge 0.
\end{aligned}
\end{equation} 
On the other hand, by the estimations for $y_{1,\gamma}$ in Propositon \ref{lem-ele-y1y2-delta-x-esti}  and for $\mathbb{S}_\gamma$ in Proposition \ref{pr-su-co-listhet}, for any $v_1, v_2 \in L_\mathbb{F}^4 (\Omega; L^4 ([0,T]; \mathbb{R}^m))$, we get 
\begin{align}\label{czlj}
&|\mathfrak{I} (v_1;\lambda^*) -\mathfrak{I} (v_2;\lambda^*)|    \notag 
\\&
\indent 
=\Big|-\int_{\Gamma} \mathbb{E}\int_0^T \left\langle\mathbb{S}_\gamma(t)y_{1,\gamma}(t;v_1), v_1(t)\right\rangle \mathrm{d}t \lambda^*(u;\mathrm{d}\gamma)   \notag 
\\& 
\indent\indent  
+\int_{\Gamma} \mathbb{E}\int_0^T \left\langle\mathbb{S}_\gamma(t)y_{1,\gamma}(t;v_2), v_2(t)\right\rangle \mathrm{d}t \lambda^*(u;\mathrm{d}\gamma)\Big|    \notag 
\\&
\indent 
\le  \int_{\Gamma} \mathbb{E} \int_0^T |\mathbb{S}_\gamma(t)| |y_{1,\gamma}(t;v_1-v_2)| |v_1| \mathrm{d}t  \lambda^*(u;\mathrm{d}\gamma)
\\&
\indent \indent 
+  \int_{\Gamma} \mathbb{E} \int_0^T |\mathbb{S}_\gamma(t)| |y_{1,\gamma}(t;v_2)| |v_1-v_2| \mathrm{d}t  \lambda^*(u;\mathrm{d}\gamma)  \notag 
\\&\indent 
\lesssim \sup_{\gamma \in\Gamma }\Big(\E \int_{0}^{T}|\mathbb{S}_\gamma(t)|^2 \dif t \Big)^{\frac{1}{2} } \Big(\E \int_{0}^{T}|v_1-v_2|^4 \dif t \Big)^{\frac{1}{4} }\Big[\Big(\E \int_{0}^{T}|v_1|^4 \dif t \Big)^{\frac{1}{4} }+\Big(\E \int_{0}^{T}|v_2|^4 \dif t \Big)^{\frac{1}{4} }\Big],  \notag 
\end{align}
where we interpolate one term $\left\langle\mathbb{S}_\gamma(t)y_{1,\gamma}(t;v_2), v_1(t)\right\rangle$ utilizing the linearity of $y$ in $v$ as shown by \eqref{lirephy}, then estimate with triangular inequality, Cauchy inequality and H\"older's inequality.
The above estimates result in the continuity of $\mathfrak{I} (v;\lambda^*)$ in $v$. 
Besides, $\Lambda$ is convex and weakly compact from (H2), and it holds similarly for $\Lambda^{\bar{u}} $. 
Furthermore, $\mathfrak{I} (\tilde{u};\cdot)$ is naturally linear with respect to $\lambda^*(\mathrm{d}\gamma) \in \Lambda^{\bar{u}}$ for any given $\tilde{u}$. 
What's more,  the monotonicity condition in (H5)
yield the conclusion that for any $u_1(\cdot), u_2(\cdot) \in L_\mathbb{F}^4 (\Omega; L^4 ([0,T]; \mathbb{R}^m))$ and $\lambda^* \in \Lambda$, 
\[\mathfrak{I} (u_1-u_2;\lambda^* ) = - \int_{\Gamma} \mathbb{E}\int_0^T \left\langle\mathbb{S}_\gamma(t)y_{1,\gamma}(t;u_1-u_2), u_1(t)-u_2(t)\right\rangle \mathrm{d}t \lambda^* (\mathrm{d}\gamma) \ge 0.\]
It implies that for any $\kappa \in(0,1)$, 
\begin{align*} 
    & (\kappa-\kappa^{2} ) \mathfrak{I} (u_{1}, u_{1} ; \lambda^{*} ) + ((1-\kappa)-(1-\kappa)^{2} ) \mathfrak{I} (u_{2}, u_{2} ; \lambda^{*} ) 
    \\  & \indent 
    \geq (\kappa-\kappa^{2} ) \mathfrak{I} (u_{1}, u_{2} ; \lambda^{*} ) + (\kappa-\kappa^{2} ) \mathfrak{I} (u_{2}, u_{1} ; \lambda^{*} ),
\end{align*}    
which further indicates that
\begin{align*}
    & 
\kappa \mathfrak{I} (u_{1}, u_{1} ; \lambda^{*} )+(1-\kappa) \mathfrak{I}\left(u_{2}, u_{2} ; \lambda^{*}\right) 
\\
&\indent 
\geq (\kappa-\kappa^{2} ) \mathfrak{I} (u_{1}, u_{2} ; \lambda^{*} )+ (\kappa-\kappa^{2} ) \mathfrak{I} (u_{2}, u_{1} ; \lambda^{*} )+\kappa^{2} \mathfrak{I} (u_{1}, u_{1} ; \lambda^{*} )+(1-\kappa)^{2} \mathfrak{I} (u_{2}, u_{2} ; \lambda^{*} )
\\
& \indent 
=\mathfrak{I} (\kappa u_{1}+(1-\kappa) u_{2}, \kappa u_{1}+(1-\kappa) u_{2} ; \lambda^{*} ) .
\end{align*}
So $\mathfrak{I} (\cdot;\lambda^*)$ is convex in $L_\mathbb{F}^4 (\Omega; L^4 ([0,T]; \mathbb{R}^m))$ for any $\lambda^* \in \Lambda^{\bar{u}}$. 
By applying the Minimax Theorem in Lemma \ref{mimmaxthe}, it deduces from \eqref{isdg} that 
$
\sup_{\lambda^*(\mathrm{d}\gamma)\in \Lambda^{\bar{u}}}    \inf_{u(\cdot)\in\mathcal{U}^4}  \mathfrak{I} (v;\lambda^*(\mathrm{d}\gamma))\ge0.
$
Now we can take $\{\lambda_N^*(\mathrm{d}\gamma)\}_{N=1}^{+\infty}\subset  \Lambda^{\bar{u}}$, such that
\[\inf_{u(\cdot)\in\mathcal{U}^4}  \mathfrak{I} (v;\lambda_N^*(\mathrm{d}\gamma))\ge -\frac{1}{N}.\]
Denote by $\lambda^*(\mathrm{d}\gamma)$ (for simplicity)  the weak limitation of $\{\lambda_N^*(\mathrm{d}\gamma)\}_{N=1}^{+\infty}$.
Then it follows that 
\[\inf_{u(\cdot)\in\mathcal{U}^4}  \mathfrak{I} (v;\lambda^*(\mathrm{d}\gamma))\ge 0.\]
That is to say, for any $v(\cdot)=u(\cdot)-\bar{u}(\cdot)$ with $u(\cdot)\in\mathcal{U}^4$, 
\begin{equation}\label{initial-second-necessary-condition}
  \begin{aligned}
\int_{\Gamma} \mathbb{E}\int_0^T \left\langle\mathbb{S}_\gamma(t)y_{1,\gamma}(t),  v(t)\right\rangle \mathrm{d}t  \lambda^*(\mathrm{d}\gamma) \le 0.
  \end{aligned}
\end{equation}
This completes the proof. 
\end{proof}

\section{The proof of Theorem~\ref{thm-point-nece-cla}}\label{seclapoisecpro}

\subsection{Technical lemmas}
We present the following two technical lemmas before carrying out the main proof. The first one proves a version of the Lebesgue differentiation theorem of multiple integrations involving both stochastic and deterministic ones under the uncertainty setting. 
 
\begin{lem}\label{lemma-cent-limi-int-tran}
Let $\lambda^*(\mathrm{d}\gamma)\in\Lambda$. Assume that $\varphi_\gamma(\cdot)  \in L_\mathbb{F}^2(\Omega;L^2(0,T;\mathbb{R}^n))$ for any $\gamma\in\Gamma$, $\varphi = \Phi, \Psi$. 
Then for a.e. $\tau\in[0,T)$,
\begin{equation}\label{lim-lem-conv-firs}
 \lim_{\alpha\rightarrow0^+} \frac{1}{\alpha^2} \int_\Gamma \mathbb{E}\int_\tau^{\tau+\alpha}
 \Big\langle\Phi_\gamma(\tau),\int_\tau^t\Psi_\gamma(s)\mathrm{d}s\Big\rangle\mathrm{d}t\lambda^*(\mathrm{d}\gamma)
 =\frac{1}{2}\int_\Gamma \mathbb{E}\Big\langle\Phi_\gamma(\tau),\Psi_\gamma(\tau)\Big\rangle \lambda^*(\mathrm{d}\gamma),
\end{equation}
\begin{equation}\label{lim-lem-conv-seco}
 \lim_{\alpha\rightarrow0^+} \frac{1}{\alpha^2} \int_\Gamma \mathbb{E}\int_\tau^{\tau+\alpha}
 \Big\langle\Phi_\gamma(t),\int_\tau^t\Psi_\gamma(s)\mathrm{d}s\Big\rangle\mathrm{d}t\lambda^*(\mathrm{d}\gamma)
 =\frac{1}{2}\int_\Gamma \mathbb{E}\Big\langle\Phi_\gamma(\tau),\Psi_\gamma(\tau)\Big\rangle \lambda^*(\mathrm{d}\gamma).
\end{equation}
\end{lem}

\begin{proof}

Since 
\begin{align*}
    \frac{1}{\alpha^2} \int_\Gamma \mathbb{E}\int_\tau^{\tau+\alpha}
    \Big\langle\Phi_\gamma(\tau),\int_\tau^t\Psi_\gamma(\tau)\mathrm{d}s\Big\rangle\mathrm{d}t\lambda^*(\mathrm{d}\gamma)
    =\frac{1}{2}\int_\Gamma \mathbb{E}\Big\langle\Phi_\gamma(\tau),\Psi_\gamma(\tau)\Big\rangle \lambda^*(\mathrm{d}\gamma) ,
\end{align*}
to prove \eqref{lim-lem-conv-firs}, 
it suffices to prove 
\begin{align}\label{mllfa}
    \lim_{\alpha\rightarrow0^+} \frac{1}{\alpha^2} \Big| \int_\Gamma \mathbb{E}\int_\tau^{\tau+\alpha}
 \Big\langle\Phi_\gamma(\tau),\int_\tau^t\Psi_\gamma(s)-\Psi_\gamma(\tau)\mathrm{d}s\Big\rangle\mathrm{d}t\lambda^*(\mathrm{d}\gamma) \Big|=0.
\end{align}
In fact, 
\begin{align*}
    &   \lim_{\alpha\rightarrow0^+} \frac{1}{\alpha^2} \Big| \int_\Gamma \mathbb{E}\int_\tau^{\tau+\alpha}
    \Big\langle\Phi_\gamma(\tau),\int_\tau^t\Psi_\gamma(s)-\Psi_\gamma(\tau)\mathrm{d}s\Big\rangle\mathrm{d}t\lambda^*(\mathrm{d}\gamma) \Big|
    \\  & \indent 
    \le 
    \lim_{\alpha\rightarrow0^+} \frac{1}{\alpha^2}  \int_\Gamma \int_\tau^{\tau+\alpha} \mathbb{E} \Big( |\Phi_\gamma(\tau)| \int_\tau^t \big|\Psi_\gamma(s)-\Psi_\gamma(\tau)\big| \mathrm{d}s  \Big) \mathrm{d}t\lambda^*(\mathrm{d}\gamma)  
    \\  &  \indent 
    \le 
    \lim_{\alpha\rightarrow0^+} \frac{1}{ \alpha^2}  \sup_{\gamma\in\Gamma} (\mathbb{E}|\Phi_\gamma(\tau)|^2)^{\frac{1}{2} }  \int_\tau^{\tau+\alpha} (t-\tau )^{\frac{1}{2} } \Big( \int_\tau^t \mathbb{E} \int_\Gamma \big|\Psi_\gamma(s)-\Psi_\gamma(\tau)\big|^2 \lambda^*(\mathrm{d}\gamma) \mathrm{d}s \Big)^{\frac{1}{2} } \mathrm{d}t  
    \\  &  \indent 
    \le 
    \frac{1}{2}\lim_{\alpha\rightarrow0^+}  \sup_{\gamma\in\Gamma} (\mathbb{E}|\Phi_\gamma(\tau)|^2)^{\frac{1}{2} } \sup_{t\in(\tau,\tau+\alpha)}\Big(\frac{1}{t-\tau}  \int_\tau^t \mathbb{E} \int_\Gamma \big|\Psi_\gamma(s)-\Psi_\gamma(\tau)\big|^2 \lambda^*(\mathrm{d}\gamma)\mathrm{d}s \Big)^{\frac{1}{2}},  
\end{align*}
where we have utilized the continuity of 
$ \int_\tau^t \mathbb{E}\int_\Gamma\big|\Psi_\gamma(s)-\Psi_\gamma(\tau)\big|^2 \lambda^*(\mathrm{d}\gamma)\mathrm{d}s$
in $t$ due to the absolutely continuous of Lebesgue integration, and utilized 
the continuity of $\Phi_\cdot(\cdot)$ in $\gamma$ in the moment sense to deal with the common null set of $\tau \in[0,T]$ with proof by contradiction to bound the supremum $\sup_{\gamma\in\Gamma} (\mathbb{E}|\Phi_\gamma(\tau)|^2)^{\frac{1}{2} }$, 
and use the existence of the limit 
\begin{align}\label{dqmtf}
    \lim_{t\to \tau^+}  \frac{1}{t-\tau }  \int_\tau^t \mathbb{E}\int_\Gamma\big|\Psi_\gamma(s)-\Psi_\gamma(\tau)\big|^2 \lambda^*(\mathrm{d}\gamma) \mathrm{d}s  =0 \quad  a.e. \  \tau \in [0,T)
\end{align}
by Lebesgue differentiation theorem, ensuring the well-posedness of 
\begin{align}\label{mosulep}
    \sup_{t\in (\tau,\tau+\alpha)} \Big\{\frac{1}{t-\tau }  \int_\tau^t \mathbb{E}\int_\Gamma\big|\Psi_\gamma(s)-\Psi_\gamma(\tau)\big|^2 \lambda^*(\mathrm{d}\gamma)\mathrm{d}s \Big\}.
\end{align}
Besides, the value of the term \eqref{mosulep} decreases as $\alpha \to 0^+$. 
Then by Lebesgue monotone convergence theorem,  \eqref{mllfa} is proved, so is \eqref{lim-lem-conv-firs}.

To prove \eqref{lim-lem-conv-seco} provided with \eqref{lim-lem-conv-firs}, it is sufficient to prove  
\begin{align*}
    & \lim_{\alpha\rightarrow0^+} \Big|\frac{1}{\alpha^2} \int_\Gamma \mathbb{E}\int_\tau^{\tau+\alpha}
     \Big\langle\Phi_\gamma(t)-\Phi_\gamma(\tau ),\int_\tau^t\Psi_\gamma(s)\mathrm{d}s\Big\rangle\mathrm{d}t\lambda^*(\mathrm{d}\gamma) \Big| =0. 
\end{align*}
Actually, by H\"{o}lder's inequality, Tonelli's theorem, and Lebesgue differentiation theorem, we have 
\begin{align*}
   & \lim_{\alpha\rightarrow0^+} \Big|\frac{1}{\alpha^2} \int_\Gamma \mathbb{E}\int_\tau^{\tau+\alpha}
    \Big\langle\Phi_\gamma(t)-\Phi_\gamma(\tau ),\int_\tau^t\Psi_\gamma(s)\mathrm{d}s\Big\rangle\mathrm{d}t\lambda^*(\mathrm{d}\gamma) \Big| 
    \\ & \quad 
    \le \lim_{\alpha\rightarrow0^+}  \frac{1}{\alpha^2} \int_\tau^{\tau+\alpha} \int_\Gamma \mathbb{E} \Big( 
    \big|\Phi_\gamma(t)-\Phi_\gamma(\tau ) \big|  \Big| \int_\tau^t\Psi_\gamma(s)\mathrm{d}s \Big|\Big)  \lambda^*(\mathrm{d}\gamma)  \mathrm{d}t 
    \\ & \quad 
    \le \lim_{\alpha\rightarrow0^+}  \frac{1}{\alpha^2} \Big(\int_\tau^{\tau+\alpha} \int_\Gamma \mathbb{E} 
    \big|\Phi_\gamma(t)-\Phi_\gamma(\tau ) \big|^2 \lambda^*(\mathrm{d}\gamma) \mathrm{d}t\Big)^{\frac{1}{2}} \Big( \int_\tau^{\tau+\alpha} \int_\Gamma \mathbb{E}  \Big| \int_\tau^t\Psi_\gamma(s)\mathrm{d}s \Big|^2 \lambda^*(\mathrm{d}\gamma) \mathrm{d}t \Big)^{\frac{1}{2} }
    \\ & \quad 
    \le \lim_{\alpha\rightarrow0^+}  \frac{1}{ \sqrt{2} \alpha} \Big(\int_\tau^{\tau+\alpha} \int_\Gamma \mathbb{E} 
    \big|\Phi_\gamma(t)-\Phi_\gamma(\tau ) \big|^2 \lambda^*(\mathrm{d}\gamma) \mathrm{d}t\Big)^{\frac{1}{2}} 
    \Big( \int_\tau^{\tau+\alpha} \int_\Gamma \mathbb{E} |\Psi_\gamma(s)|^2 \lambda^*(\mathrm{d}\gamma)\mathrm{d}s  \Big)^{\frac{1}{2} }
    \\ & \quad 
    =0. 
\end{align*}
\end{proof}

The following lemma plays a role in approximating the Malliavin derivative, essentially utilizing the continuity of $\mathcal{D}_\cdot\varphi_\gamma(\cdot)$ on some neighborhood of $\{(t,t)|t\in[0,T]\}$. 
\begin{lem}\label{lem-der-phi-nab-phi=0}
For $\varphi (\cdot )\in \mathbb{L}_{2,\mathbb{F}}^{1,2}(\mathbb{R}^n)$, there exists a sequence of $\{\alpha_l\}_{l=1}^{\infty}$ of positive numbers satisfying $\alpha_l\to0$ as $l\to\infty$, such that
\begin{equation}\label{der-phi-nab-phi=0}
\begin{aligned}
& \lim_{l\to\infty} \frac{1}{\alpha_l^2}\int_{\Gamma} \int_\tau^{\tau+\alpha_l}\int_\tau^t\mathbb{E}   \left|\mathcal{D}_s\varphi_\gamma(t)-\nabla\varphi_\gamma(s)\right|^2 \mathrm{d}s\mathrm{d}t \lambda^*(\mathrm{d}\gamma)=0 \quad  a.e. \ \tau\in[0,T].
\end{aligned}
\end{equation}
\end{lem}
\begin{proof}

By exchange the order of integrals, displacement of intervals, substitution of variables, it can be proven by direct calculus (cf.  \cite[Lemma 2.1]{zhanghaisenzhangxu2015siam}) that for any $\gamma\in\Gamma$,
\begin{align}\label{der-phi-nab-phi=0pf1}
& \lim_{\alpha\to0^+} \frac{1}{\alpha^2}\int_0^T \int_\tau^{\tau+\alpha}\int_\tau^t\mathbb{E}   \left|\mathcal{D}_s\varphi_\gamma(t)-\nabla\varphi_\gamma(s)\right|^2 \mathrm{d}s\mathrm{d}t\mathrm{d}\tau   \notag  \\
&\indent  \le \lim_{\alpha\to0^+}\int_0^T \sup_{t\in[\tau,\tau+\alpha]\cap[0,T]}\mathbb{E}\left|\mathcal{D}_\tau\varphi_\gamma(t)-\nabla\varphi_\gamma(\tau)\right|^2 \mathrm{d}\tau   \\
&\indent =0,  \notag 
\end{align}
where the last equality is from the definition of $\nabla\varphi_\gamma(\cdot)$.
Then it follows 
\begin{subequations}\label{der-phi-nabruiuy-phi=0pf1}
\begin{align}
& \lim_{\alpha\to0^+}\frac{1}{\alpha^2}\int_\Gamma \int_0^T \int_\tau^{\tau+\alpha}\int_\tau^t\mathbb{E}   \left|\mathcal{D}_s\varphi_\gamma(t)-\nabla\varphi_\gamma(s)\right|^2 \mathrm{d}s\mathrm{d}t\mathrm{d}\tau \lambda^*(\mathrm{d}\gamma)  \notag \\
&\indent  \le \lim_{\alpha\to0^+}\int_\Gamma\int_0^T \sup_{t\in[\tau,\tau+\alpha]\cap[0,T]}\mathbb{E}\left|\mathcal{D}_\tau\varphi_\gamma(t)-\nabla\varphi_\gamma(\tau)\right|^2 \mathrm{d}\tau \lambda^*(\mathrm{d}\gamma)  \label{der-phi-nabruiuy-phi=0pf1a}\\
&\indent   = \int_\Gamma\lim_{\alpha\to0^+}\int_0^T \sup_{t\in[\tau,\tau+\alpha]\cap[0,T]}\mathbb{E}\left|\mathcal{D}_\tau\varphi_\gamma(t)-\nabla\varphi_\gamma(\tau)\right|^2 \mathrm{d}\tau \lambda^*(\mathrm{d}\gamma)  \label{der-phi-nabruiuy-phi=0pf1b}  \\
&\indent =0,   \label{der-phi-nabruiuy-phi=0pf1c}
\end{align}
\end{subequations}
where \eqref{der-phi-nabruiuy-phi=0pf1a} and \eqref{der-phi-nabruiuy-phi=0pf1c} are from~\eqref{der-phi-nab-phi=0pf1}, 
while \eqref{der-phi-nabruiuy-phi=0pf1b} is from the Lebesgue monotone convergence theorem, 
since 
\[\int_0^T\mathop{\sup}_{t\in[\tau,\tau+\alpha]\cap[0,T]}\mathbb{E}\left|\mathcal{D}_\tau\varphi_\gamma(t)-\nabla\varphi_\gamma(\tau)\right|^2 \mathrm{d}\tau\]
is non-negative and decreases as $\alpha\to0^+$ for any $\gamma\in\Gamma$.
By~\eqref{der-phi-nabruiuy-phi=0pf1} and Fubini's theorem, we obtain  
\begin{equation*}
\begin{aligned}
& \lim_{\alpha\to0^+}\frac{1}{\alpha^2} \int_0^T \int_\Gamma\int_\tau^{\tau+\alpha}\int_\tau^t\mathbb{E}   \left|\mathcal{D}_s\varphi_\gamma(t)-\nabla\varphi_\gamma(s)\right|^2 \mathrm{d}s\mathrm{d}t \lambda^*(\mathrm{d}\gamma)\mathrm{d}\tau =0,
\end{aligned}
\end{equation*}
then~\eqref{der-phi-nab-phi=0} follows.
\end{proof}

\subsection{The proof of Theorem ~\ref{thm-point-nece-cla}}

\begin{proof}[The proof of Theorem~\ref{thm-point-nece-cla}]
Now we derive the pointwise necessary conditions from the integral type by performing density arguments, thoroughly including temporal, spatial, and sample variables. Additionally, we adopt methods proposed in \cite{zhanghaisenzhangxu2015siam} and extended in \cite{zhanghaisenzhangxu2018siamreview,luqizhanghaisenzhangxu2021siam} with tools from Malliavin calculus to address subtle problems related to the integrable order deficiencies in multiple integrals of Lebesgue and It\^o types. These deficiencies hinder the effectiveness of the classical Lebesgue differentiation theorem. 
Throughout this process, we will repeatedly utilize the above two lemmas of Lebesgue differentiation and Malliavin approximation tailored to our model. 

Since $\mathbb{F}$ is the natural filtration generated by the standard Brownian motion augmented by all the $\mathbb{P}$-null sets, one can find a sequence $\{A_l\}_{l=1}^{\infty}\subset\mathcal{F}_t, \forall t\in[0,T]$, such that for any $A\in\mathcal{F}_t$, there exists a sub-sequence $\{A_{l_i}\}_{i=1}^{\infty}\subset\{A_l\}_{l=1}^{\infty}$,
satisfying
$\lim_{i\to\infty}\mathbb{P}(A\vartriangle A_{l_i})=0$,
where $A\vartriangle A_{l_i}:=(A\backslash A_{l_i})\cup(A_{l_i}\backslash A)$.
Denote by $\{t_i\}_{i=1}^\infty$ the sequence of rational numbers in $[0,T)$, by $\{v^k\}_{k=1}^\infty$ a dense subset of $U$. For any $i\in \mathbb{N}$, take a sequence $\{A_{ij}\}_{j=1}^{\infty}(\subset\mathcal{F}_{t_i})$ which can generate $\mathcal{F}_{t_i}$.
Fix $i,j,k$ arbitrarily.
For any $\tau\in[t_i,T)$ and $\alpha\in(0,T-\tau)$, denote $E_\alpha^i=[\tau,\tau+\alpha)$,
\begin{equation*}
u_{ij}^k(t,\omega)=
\left\{
\begin{aligned}
   v^k,  \indent\quad& (t,\omega)\in  E_\alpha^i\times A_{ij},   \\
\bar{u}(t,\omega),  \indent &  (t,\omega)\in ([0,T]\times\Omega)\backslash (E_\alpha^i\times A_{ij}),
\end{aligned}
\right.
\end{equation*}
and
\[v_{ij}^k(t,\omega)=u_{ij}^k(t,\omega)-\bar{u}(t,\omega) =(v^k-\bar{u}(t,\omega))\chi_{A_{ij}}(\omega)\chi_{E_\alpha^i}(t),\quad (t,\omega)\in[0,T]\times\Omega.\]
Clearly $u_{ij}^k(\cdot), v_{ij}^k(\cdot)\in\mathcal{U}^4 $.
Denote by $y_{\gamma}^{ijk}$ the solution to~\eqref{variant-equat-first} with $v(\cdot)$ replaced by $v_{ij}^k(\cdot)$, recalling \eqref{lirephy}, then actually
\begin{equation}\label{explicit-form-yijk}
\begin{aligned}
y_{\gamma}^{ijk}(t)=&\Phi_\gamma(t)\int_0^t \Phi_\gamma(s)^{-1}\left(\partial_u b_\gamma(s)-\partial_x\sigma_\gamma(s)\partial_u\sigma_\gamma(s)\right) (v^k-\bar{u}(s))\chi_{A_{ij}}(\omega)\chi_{E_\alpha^i}(s)\mathrm{d}s   \\
&+\Phi_\gamma(t)\int_0^t\Phi_\gamma(s)^{-1}\partial_u\sigma_\gamma(s) (v^k-\bar{u}(s))\chi_{A_{ij}}(\omega)\chi_{E_\alpha^i}(s)\mathrm{d}W(s).
\end{aligned}
\end{equation}
Replacing $v(\cdot)$ by $v_{ij}^k(\cdot)$ in~\eqref{initial-second-necessary-condition} results in
\begin{equation}\label{initial-second-necessary-condition-yijk}
\int_{\Gamma}  \mathbb{E}\int_\tau^{\tau+\alpha} \left\langle\mathbb{S}_\gamma(t)y^{ijk}_{\gamma}(t), v^k-\bar{u}(t) \right\rangle \mathrm{d}t  \lambda^*(\mathrm{d}\gamma)\le0.
\end{equation}
Then plug \eqref{explicit-form-yijk} into~\eqref{initial-second-necessary-condition-yijk} resulting in 
\[0\ge I_1(\alpha )+I_0(\alpha ),\]
where 
\begin{align*}
& I_1(\alpha )=\frac{1}{\alpha^2}\int_{\Gamma} \mathbb{E}\int_\tau^{\tau+\alpha} \Big\langle\mathbb{S}_\gamma(t)\Phi_\gamma(t)\int_\tau^t \Phi_\gamma(s)^{-1}\left(\partial_u b_\gamma(s)-\partial_x\sigma_\gamma(s)\partial_u\sigma_\gamma(s)\right) 
\\&\indent\indent\indent\indent\indent\indent \cdot (v^k-\bar{u}(s))\chi_{A_{ij}}(\omega)\mathrm{d}s,  v^k-\bar{u}(t) \Big\rangle \chi_{A_{ij}}(\omega)\mathrm{d}t  \lambda^*(\mathrm{d}\gamma),   
\\&
I_0(\alpha )= \frac{1}{\alpha^2}\int_{\Gamma} \mathbb{E}\int_\tau^{\tau+\alpha} \Big\langle\mathbb{S}_\gamma(t)\Phi_\gamma(t)\int_\tau^t \Phi_\gamma(s)^{-1}\partial_u\sigma_\gamma(s) 
\\&\indent\indent\indent\indent\indent\indent \cdot (v^k-\bar{u}(s))\chi_{A_{ij}}(\omega)\mathrm{d}W(s), v^k-\bar{u}(t) \Big\rangle \chi_{A_{ij}}(\omega)\mathrm{d}t  \lambda^*(\mathrm{d}\gamma).
\end{align*}
$I_0(\alpha )$ is in the form of a combination of Lebesgue and It\^o multiple integration, preventing a good use of Lebesgue differentiation theorem. 
Recalling~\eqref{solu-form-ele-phi}, replacing the $\Phi_\gamma(t)$ in  $I_0(\alpha )$ by its explicit representation, we arrive at 
\[0 \ge  I_1(\alpha )+I_2(\alpha )+I_3(\alpha )+I_4(\alpha ),\]
where 
\begin{align*}
& 
I_2(\alpha )=\frac{1}{\alpha^2}\int_{\Gamma}\int_\tau^{\tau+\alpha}\mathbb{E} \Big[ \Big\langle\mathbb{S}_\gamma(t)\Phi_\gamma(\tau)\int_\tau^t \Phi_\gamma(s)^{-1}\partial_u\sigma_\gamma(s) 
\\&\indent\indent\indent\indent\indent\indent \cdot (v^k-\bar{u}(s))\chi_{A_{ij}}(\omega)\mathrm{d}W(s),v^k-\bar{u}(t) \Big\rangle \chi_{A_{ij}}(\omega)\Big]\mathrm{d}t \lambda^*(\mathrm{d}\gamma), 
\\
& 
I_3(\alpha )= \frac{1}{\alpha^2}\int_{\Gamma}\int_\tau^{\tau+\alpha} \mathbb{E} \Big[ \Big\langle\mathbb{S}_\gamma(t)\int_\tau^t\partial_x b_\gamma(s)\Phi_\gamma(s)\mathrm{d}s\int_\tau^t \Phi_\gamma(s)^{-1}\partial_u\sigma_\gamma(s) \\&\indent\indent\indent\indent\indent\indent \cdot (v^k-\bar{u}(s))\chi_{A_{ij}}(\omega)\mathrm{d}W(s),v^k-\bar{u}(t) \Big\rangle \chi_{A_{ij}}(\omega)\Big]\mathrm{d}t \lambda^*(\mathrm{d}\gamma),   
\\
& 
I_4 (\alpha )= \frac{1}{\alpha^2}\int_{\Gamma}\int_\tau^{\tau+\alpha} \mathbb{E} \Big[ \Big\langle\mathbb{S}_\gamma(t)\int_\tau^t\partial_x \sigma_\gamma(s)\Phi_\gamma(s)\mathrm{d}W(s)\int_\tau^t \Phi_\gamma(s)^{-1}\partial_u\sigma_\gamma(s) 
\\&\indent\indent\indent\indent\indent\indent \cdot (v^k-\bar{u}(s))\chi_{A_{ij}}(\omega)\mathrm{d}W(s),v^k-\bar{u}(t) \Big\rangle \chi_{A_{ij}}(\omega)\Big]\mathrm{d}t \lambda^*(\mathrm{d}\gamma).
\end{align*}
Among the remaining  five parts of the proof,   we will estimate the above $I_1(\alpha ), I_2(\alpha ), I_3(\alpha ), I_4(\alpha )$ as ${\alpha\to0^+}$ separately and complete the proof ultimately.

\textbf{Part 1:}  \  Since $I_1(\alpha )$ is a term of multiple Lebesgue integration,  from (H3) (ii), together with  Lemma~\ref{lemma-cent-limi-int-tran}, we obtain 
\begin{equation*}
\begin{aligned}
&\lim_{\alpha\to0^+} I_1(\alpha ) =\frac{1}{2} \int_{\Gamma}\mathbb{E}\Big[ \Big\langle\mathbb{S}_\gamma(\tau) \left(\partial_u b_\gamma(\tau)-\partial_x\sigma_\gamma(\tau)\partial_u\sigma_\gamma(\tau)\right) \\&\indent\indent\indent\indent\indent\indent \cdot (v^k-\bar{u}(\tau)),v^k-\bar{u}(\tau) \Big\rangle \chi_{A_{ij}}(\omega) \Big]\lambda^*(\mathrm{d}\gamma) \ \ \  a.e. \  \tau\in[t_i,T).
\end{aligned}
\end{equation*}

\textbf{Part 2:}   \   In this part we estimate $I_2(\alpha )$. Based on (H6) (i), by the Clark-Ocone formula~\eqref{claocoori}, for a.e. $t\in[0,T]$ and any $\gamma\in\Gamma$,
$
\mathbb{S}_\gamma(t)^\top(v^k-\bar{u}(t)) =\mathbb{E} [\mathbb{S}_\gamma(t)^\top(v^k-\bar{u}(t)) ]+\int_0^t \mathbb{E} [\mathcal{D}_s (\mathbb{S}_\gamma(t)^\top(v^k-\bar{u}(t)) )|\mathcal{F}_s ] \mathrm{d}W(s).
$
Substitute this representation  into $I_2$, then by elementary martingale analysis, we get 
\begin{align*}
    & 
    I_2(\alpha )= 
    \frac{1}{\alpha^2}\int_{\Gamma} \int_\tau^{\tau+\alpha}\int_\tau^t\mathbb{E}\Big[ \Big\langle \Phi_\gamma(\tau)  \Phi_\gamma(s)^{-1}\partial_u\sigma_\gamma(s) (v^k-\bar{u}(s)), 
    \\
&\indent\indent\indent\indent\indent\indent     \mathcal{D}_s \Big( \mathbb{S}_\gamma(t)^\top (v^k-\bar{u}(t)) \Big) \Big\rangle \chi_{A_{ij}}(\omega)\Big]\mathrm{d}s\mathrm{d}t \lambda^*(\mathrm{d}\gamma) . 
\end{align*}
Recalling that on the basis of  (H6), 
$
\mathcal{D}_s\left(\mathbb{S}_\gamma(t)^\top(v^k-\bar{u}(t))\right) =\mathcal{D}_s\mathbb{S}_\gamma(t)^\top(v^k-\bar{u}(t)) -\mathbb{S}_\gamma(t)^\top\mathcal{D}_s\bar{u}(t).
$
Then 
\begin{align*}
&  I_2 (\alpha ) =   \frac{1}{\alpha^2}\int_{\Gamma} \int_\tau^{\tau+\alpha}\int_\tau^t\mathbb{E}\Big[ \Big\langle \Phi_\gamma(\tau)  \Phi_\gamma(s)^{-1}\partial_u\sigma_\gamma(s) (v^k-\bar{u}(s)), \\
&\indent\indent\indent\indent\indent\indent     \mathcal{D}_s\mathbb{S}_\gamma(t)^\top (v^k-\bar{u}(t))\Big\rangle \chi_{A_{ij}}(\omega)\Big]\mathrm{d}s\mathrm{d}t \lambda^*(\mathrm{d}\gamma)  \\
&\indent\indent  - \frac{1}{\alpha^2}\int_{\Gamma} \int_\tau^{\tau+\alpha}\int_\tau^t\mathbb{E}\Big[ \Big\langle \Phi_\gamma(\tau)  \Phi_\gamma(s)^{-1}\partial_u\sigma_\gamma(s) (v^k-\bar{u}(s)), \\ & \indent\indent\indent\indent\indent\indent     \mathbb{S}_\gamma(t)^\top\mathcal{D}_s\bar{u}(t)\Big\rangle \chi_{A_{ij}}(\omega)\Big]\mathrm{d}s\mathrm{d}t \lambda^*(\mathrm{d}\gamma)   \\
&\quad =:I_5(\alpha )-I_6(\alpha ).
\end{align*}
Now we approximate $\mathcal{D}_\cdot\mathbb{S}_\gamma(\cdot)$ by $\nabla\mathbb{S}_\gamma(\cdot)$ for any $\gamma\in\Gamma$, to utilize the continuity of $\mathcal{D}_\cdot\varphi_\gamma(\cdot)$ on some neighborhood of $\{(t,t)|t\in[0,T]\}$:
\begin{equation*}
\begin{aligned}
& I_5 (\alpha ) =\frac{1}{\alpha^2}\int_{\Gamma} \int_\tau^{\tau+\alpha}\int_\tau^t\mathbb{E}\Big[ \Big\langle \Phi_\gamma(\tau)  \Phi_\gamma(s)^{-1}\partial_u\sigma_\gamma(s) (v^k-\bar{u}(s)), \\
&\indent\indent\indent\indent\indent\indent     \left(\mathcal{D}_s\mathbb{S}_\gamma(t)-\nabla\mathbb{S}_\gamma(s)\right)^\top (v^k-\bar{u}(t))\Big\rangle \chi_{A_{ij}}(\omega)\Big]\mathrm{d}s\mathrm{d}t \lambda^*(\mathrm{d}\gamma)  \\
&\indent\indent  + \frac{1}{\alpha^2}\int_{\Gamma} \int_\tau^{\tau+\alpha}\int_\tau^t\mathbb{E}\Big[ \Big\langle \Phi_\gamma(\tau)  \Phi_\gamma(s)^{-1}\partial_u\sigma_\gamma(s) (v^k-\bar{u}(s)), \\
&\indent\indent\indent\indent\indent\indent     \nabla\mathbb{S}_\gamma(s)^\top (v^k-\bar{u}(t))\Big\rangle \chi_{A_{ij}}(\omega)\Big]\mathrm{d}s\mathrm{d}t \lambda^*(\mathrm{d}\gamma)  \\&\quad  =:I_7(\alpha )+I_8(\alpha ).
\end{aligned}
\end{equation*}
By Lemma~\ref{lem-der-phi-nab-phi=0}, there exists a sequence of ${\{\alpha_l\}}_{l=1}^{\infty}$ satisfying $\alpha_l\to0$ as $l\to\infty$, such that
\begin{equation*}
  \begin{aligned}
  & \lim_{l\to\infty}I_7(\alpha_l)=\lim_{l\to\infty} \frac{1}{\alpha_l^2}\int_{\Gamma} \int_\tau^{\tau+\alpha_l} \!\!\! \int_\tau^t\mathbb{E}\Big[ \Big\langle \Phi_\gamma(\tau)  {\Phi_\gamma(s)}^{-1}\partial_u\sigma_\gamma(s) (v^k-\bar{u}(s)), \\
&\indent\indent\indent\indent\indent  \left(\mathcal{D}_s\mathbb{S}_\gamma(t)-\nabla\mathbb{S}_\gamma(s)\right)^\top (v^k-\bar{u}(t))\Big\rangle \chi_{A_{ij}}(\omega)\Big]\mathrm{d}s\mathrm{d}t \lambda^*(\mathrm{d}\gamma)=0 \ \  a.e. \  \tau\in[t_i,T).
\end{aligned}
\end{equation*}
Besides, by (H6) and Lemma~\ref{lemma-cent-limi-int-tran}, it holds that 
\begin{equation*}
\begin{aligned}
\lim_{\alpha\to0^+} I_8(\alpha )=  \frac{1}{2} \int_{\Gamma}  \mathbb{E}\Big[ \Big\langle  \nabla\mathbb{S}_\gamma(\tau)  \partial_u\sigma_\gamma(\tau) (v^k-\bar{u}(\tau)),   v^k-\bar{u}(\tau)\Big\rangle \chi_{A_{ij}}(\omega)\Big] \lambda^*(\mathrm{d}\gamma)\ \  a.e. \  \tau\in[t_i,T).
\end{aligned}
\end{equation*}
Then
\begin{equation*}
\begin{aligned}
\lim_{l\to\infty}I_5(\alpha_l)=  \frac{1}{2} \int_{\Gamma}  \mathbb{E}\Big[ \Big\langle  \nabla\mathbb{S}_\gamma(\tau)  \partial_u\sigma_\gamma(\tau) (v^k-\bar{u}(\tau)),   v^k-\bar{u}(\tau)\Big\rangle \chi_{A_{ij}}(\omega)\Big] \lambda^*(\mathrm{d}\gamma)\ \  a.e. \  \tau\in[t_i,T).
\end{aligned}
\end{equation*}

Similarly, approximating $\mathcal{D}_\cdot\bar{u}(\cdot)$ by $\nabla\bar{u}(\cdot)$,  there exists a sub-sequence of the above sequence  ${\{\alpha_l\}}_{l=1}^{\infty}$, still denoted by itself for simplicity, such that
\begin{equation*}
\begin{aligned}
\lim_{l\to\infty}I_6(\alpha_l)=  \frac{1}{2} \int_{\Gamma}\mathbb{E} \Big[ \Big\langle  \mathbb{S}_\gamma(\tau) \partial_u\sigma_\gamma(\tau) (v^k-\bar{u}(\tau)),  \nabla\bar{u}(\tau)\Big\rangle \chi_{A_{ij}}(\omega)\Big] \lambda^*(\mathrm{d}\gamma)\quad   a.e. \  \tau\in[t_i,T).
\end{aligned}
\end{equation*}

Then there is a sequence  ${\{\alpha_l\}}_{l=1}^{\infty}$, such that
\begin{equation*}
\begin{aligned}
&\lim_{l\to\infty}I_2(\alpha_l)=  \frac{1}{2} \int_{\Gamma}  \mathbb{E}\Big[ \Big\langle  \nabla\mathbb{S}_\gamma(\tau)  \partial_u\sigma_\gamma(\tau) (v^k-\bar{u}(\tau)),   v^k-\bar{u}(\tau)\Big\rangle \chi_{A_{ij}}(\omega)\Big] \lambda^*(\mathrm{d}\gamma)  \\
&\indent\indent\indent\indent -\frac{1}{2} \int_{\Gamma}\mathbb{E} \Big[ \Big\langle  \mathbb{S}_\gamma(\tau) \partial_u\sigma_\gamma(\tau) (v^k-\bar{u}(\tau)),  \nabla\bar{u}(\tau)\Big\rangle \chi_{A_{ij}}(\omega)\Big] \lambda^*(\mathrm{d}\gamma)  \quad   
 a.e. \  \tau\in[t_i,T).
\end{aligned}
\end{equation*}

\textbf{Part 3:}   \
Denote 
$I_3(\alpha )=:\int_{\Gamma} I_3^0(\alpha ;\gamma ) \lambda^*(\mathrm{d}\gamma)$.
From H\"{o}lder's  inequality,  Burkholder-Davis-Gundy inequality, and $\overline{u}(\cdot)\in \mathcal{U}^4 $, it can be verified that for any  $\gamma\in\Gamma$,
\begin{align*}
    |I_3^0(\alpha ;\gamma )|
&=\Big|\frac{1}{\alpha^2} \int_\tau^{\tau+\alpha} \mathbb{E} \Big[ \Big\langle\mathbb{S}_\gamma(t)\int_\tau^t\partial_x b_\gamma(s)\Phi_\gamma(s)\mathrm{d}s\int_\tau^t \Phi_\gamma(s)^{-1}\partial_u\sigma_\gamma(s)   \\
&\indent\indent\indent\indent\indent\indent \cdot (v^k-\bar{u}(s))\chi_{A_{ij}}(\omega)\mathrm{d}W(s),v^k-\bar{u}(t)\Big\rangle \chi_{A_{ij}}(\omega) \Big] \mathrm{d}t \Big|  \\
&   \lesssim \frac{1}{\alpha^2} \int_\tau^{\tau+\alpha} (t-\tau)^{\frac{3}{2}} \big[\mathbb{E}\left|\mathbb{S}_\gamma(t)\right|^2\big]^{\frac{1}{2}}\mathrm{d}t,
\end{align*}
and
\begin{equation*}
\begin{aligned}
\lim_{\alpha\to0^+} \frac{1}{\alpha^2} \int_\tau^{\tau+\alpha} (t-\tau)^{\frac{3}{2}} \big[\mathbb{E}\left|\mathbb{S}_\gamma(t)\right|^2\big]^{\frac{1}{2}}\mathrm{d}t=0 \quad  a.e. \  \tau\in[t_i,T),\ \gamma\in\Gamma.
\end{aligned}
\end{equation*}
Besides, by the estimate \eqref{jhzls},
$
\lim_{\ell \to 0}\sup_{ \mathsf{d} (\gamma,\gamma')\le\ell }\int_0^T\mathbb{E} |\mathbb{S}_\gamma(t)-\mathbb{S}_{\gamma'}(t) |^2\mathrm{d}t=0.
$
Then actually
\begin{equation*}\label{es-sthe-cont-conc1}
\begin{aligned}
\lim_{\alpha\to0^+} \frac{1}{\alpha^2} \int_\tau^{\tau+\alpha} (t-\tau)^{\frac{3}{2}} \big[\mathbb{E}\left|\mathbb{S}_\gamma(t)\right|^2\big]^{\frac{1}{2}}\mathrm{d}t=0 \quad   \forall \gamma\in\Gamma,\  a.e. \  \tau\in[t_i,T).
\end{aligned}
\end{equation*}
By H\"{o}lder's inequality and Tonelli's theorem, we have 
\begin{align*}
\lim_{\alpha\to0^+}I_3 (\alpha ) 
&
=\lim_{\alpha\to0^+} \int_{\Gamma} I_3^0(\alpha ;\gamma ) \lambda^*(\mathrm{d}\gamma)
\lesssim \lim_{\alpha\to0^+}  \frac{1}{\alpha^2} \int_\Gamma \int_\tau^{\tau+\alpha} (t-\tau)^{\frac{3}{2}} \big[\mathbb{E}\left|\mathbb{S}_\gamma(t)\right|^2\big]^{\frac{1}{2}}\mathrm{d}t \lambda^*(\mathrm{d}\gamma)
\\ &  
\lesssim \lim_{\alpha\to0^+} \Big[ \int_\tau^{\tau+\alpha} \int_\Gamma \mathbb{E}\left|\mathbb{S}_\gamma(t)\right|^2 \lambda^*(\mathrm{d}\gamma) \mathrm{d}t \Big]^{\frac{1}{2}} 
=0 \quad   a.e. \  \tau\in[t_i,T).
\end{align*}

\textbf{Part 4:}  \
For $I_4(\alpha )$ we firstly insert two  auxiliary items and see that  
\begin{align*}
I_4(\alpha )&=\frac{1}{\alpha^2}\int_{\Gamma} \int_\tau^{\tau+\alpha} \mathbb{E}\Big[ \Big\langle\int_\tau^t\partial_x \sigma_\gamma(s)\Phi_\gamma(s)\mathrm{d}W(s)\int_\tau^t \Phi_\gamma(s)^{-1}\partial_u\sigma_\gamma(s)(v^k-\bar{u}(s))\chi_{A_{ij}}(\omega)\mathrm{d}W(s), \\&\indent\indent\indent\indent\indent  \indent  \mathbb{S}_\gamma(t)^\top(v^k-\bar{u}(t))-\mathbb{S}_\gamma(\tau)^\top(v^k-\bar{u}(\tau))\Big\rangle \chi_{A_{ij}}(\omega)\Big]\mathrm{d}t \lambda^*(\mathrm{d}\gamma)  
\\
&\indent  +\frac{1}{\alpha^2}\int_{\Gamma} \int_\tau^{\tau+\alpha} \mathbb{E}\Big[ \Big\langle\mathbb{S}_\gamma(\tau)\int_\tau^t\partial_x \sigma_\gamma(s)\Phi_\gamma(s)\mathrm{d}W(s)\int_\tau^t \Phi_\gamma(s)^{-1}\partial_u\sigma_\gamma(s) 
\\&\indent\indent\indent\indent\indent\indent \cdot (v^k-\bar{u}(s))\chi_{A_{ij}}(\omega)\mathrm{d}W(s),v^k-\bar{u}(\tau)\Big\rangle \chi_{A_{ij}}(\omega)\Big]\mathrm{d}t \lambda^*(\mathrm{d}\gamma)   \\
&   =:I_9(\alpha )+I_{10}(\alpha ).
\end{align*}
For any $\gamma\in\Gamma$, by calculus of  It\^o integration and H\"{o}lder's inequality, it can be directly verified that 
\begin{align*}
&   \frac{1}{\alpha^2} \Big| \int_\tau^{\tau+\alpha} \mathbb{E} \Big\langle\int_\tau^t\partial_x \sigma_\gamma(s)\Phi_\gamma(s)\mathrm{d}W(s)\int_\tau^t \Phi_\gamma(s)^{-1}\partial_u\sigma_\gamma(s)(v^k-\bar{u}(s))\chi_{A_{ij}}(\omega)\mathrm{d}W(s), \\&\indent\indent\indent\indent\indent  \indent  \mathbb{S}_\gamma(t)^\top(v^k-\bar{u}(t))-\mathbb{S}_\gamma(\tau)^\top(v^k-\bar{u}(\tau))\Big\rangle \chi_{A_{ij}}(\omega)\mathrm{d}t \Big|  \\
&\indent  \lesssim \frac{1}{\alpha^{\frac{1}{2}}}\Big[\int_\tau^{\tau+\alpha} \mathbb{E}\left|\mathbb{S}_\gamma(t)^\top(v^k-\bar{u}(t))-\mathbb{S}_\gamma(\tau)^\top(v^k-\bar{u}(\tau))\right|^2 \mathrm{d}t\Big]^{\frac{1}{2}},
\end{align*}
and further by \eqref{jhzls} and Tonelli's theorem, we have 
\begin{align*}
    \lim_{\alpha\to0^+}I_9(\alpha )& \lesssim  
\lim_{\alpha\to0^+}\int_{\Gamma}\frac{1}{\alpha^{\frac{1}{2}}}\Big[\int_\tau^{\tau+\alpha} \mathbb{E}\left|\mathbb{S}_\gamma(t)^\top(v^k-\bar{u}(t))-\mathbb{S}_\gamma(\tau)^\top(v^k-\bar{u}(\tau))\right|^2 \mathrm{d}t\Big]^{\frac{1}{2}} \lambda^*(\mathrm{d}\gamma) 
\\  &  \le  
\lim_{\alpha\to0^+} \Big[\frac{1}{\alpha} \int_\tau^{\tau+\alpha} \int_{\Gamma} \mathbb{E}\left|\mathbb{S}_\gamma(t)^\top(v^k-\bar{u}(t))-\mathbb{S}_\gamma(\tau)^\top(v^k-\bar{u}(\tau))\right|^2   \lambda^*(\mathrm{d}\gamma) \mathrm{d}t \Big]^{\frac{1}{2}}
\\  &  
=0 \quad   a.e. \  \tau\in[t_i,T).
\end{align*}
Besides, by calculus of  It\^o integration again and Lemma~\ref{lemma-cent-limi-int-tran}, for $a.e. \  \tau\in[t_i,T)$, it follows 
\begin{align*}
& \lim_{\alpha\to0^+}I_{10}(\alpha ) =\frac{1}{2}\int_{\Gamma} \mathbb{E}\Big[ \Big\langle\mathbb{S}_\gamma(\tau) \partial_x \sigma_\gamma(\tau) \partial_u\sigma_\gamma(\tau)   
(v^k-\bar{u}(\tau)), v^k-\bar{u}(\tau)\Big\rangle \chi_{A_{ij}}(\omega) \Big]\lambda^*(\mathrm{d}\gamma)  .
\end{align*}
Then for $a.e. \  \tau\in[t_i,T)$, we obtain 
\begin{align*}
\lim_{\alpha\to0^+}I_4(\alpha )&=\lim_{\alpha\to0^+}(I_9(\alpha )+I_{10}(\alpha )) \\
&=\frac{1}{2}\int_{\Gamma}\mathbb{E}\Big[ \Big\langle\mathbb{S}_\gamma(\tau) \partial_x \sigma_\gamma(\tau) \partial_u\sigma_\gamma(\tau)  (v^k-\bar{u}(\tau)), v^k-\bar{u}(\tau)\Big\rangle \chi_{A_{ij}}(\omega) \Big]\lambda^*(\mathrm{d}\gamma) .
\end{align*}

\textbf{Part 5:}   \ By the above estimations for $I_1(\alpha ), I_2(\alpha ), I_3(\alpha ), I_4(\alpha )$,  for the arbitrarily fixed $i,j,k$ in the beginning of the proof, there exists a Lebesgue measurable set $E_{i,j}^k\subset[t_i,T)$ satisfying $|E_{i,j}^k|=0$, such that
\begin{align*}
0&\ge \lim_{\alpha\to0^+}\(I_1(\alpha )+I_2(\alpha )+I_3(\alpha )+I_4(\alpha ) \) \\
& =\frac{1}{2} \int_{\Gamma}\mathbb{E}\Big[ \Big\langle\mathbb{S}_\gamma(\tau) \partial_u b_\gamma(\tau) (v^k-\bar{u}(\tau)),v^k-\bar{u}(\tau)\Big\rangle \chi_{A_{ij}}(\omega) \Big]\lambda^*(\mathrm{d}\gamma)  \\
&\quad  +\frac{1}{2} \int_{\Gamma} \mathbb{E}\Big[ \Big\langle  \nabla\mathbb{S}_\gamma(\tau)  \partial_u\sigma_\gamma(\tau) (v^k-\bar{u}(\tau)),   v^k-\bar{u}(\tau)\Big\rangle \chi_{A_{ij}}(\omega)\Big] \lambda^*(\mathrm{d}\gamma)  \\
&\quad  -\frac{1}{2} \int_{\Gamma}\mathbb{E} \Big[ \Big\langle  \mathbb{S}_\gamma(\tau) \partial_u\sigma_\gamma(\tau) (v^k-\bar{u}(\tau)),  \nabla\bar{u}(\tau)\Big\rangle \chi_{A_{ij}}(\omega)\Big] \lambda^*(\mathrm{d}\gamma) \quad   \forall\tau\in[t_i,T)\backslash E_{i,j}^k.
\end{align*}
Then by the density of $\{v^k\}_{k=1}^\infty$, it deduces
\begin{align*}
&\int_{\Gamma}\mathbb{E}\Big[ \Big\langle\mathbb{S}_\gamma(\tau) \partial_u b_\gamma(\tau) (v-\bar{u}(\tau)),v-\bar{u}(\tau)\Big\rangle \chi_{A_{ij}}(\omega) \Big] \lambda^*(\mathrm{d}\gamma)  \\
&\indent  +\int_{\Gamma} \mathbb{E}\Big[ \Big\langle  \nabla\mathbb{S}_\gamma(\tau)  \partial_u\sigma_\gamma(\tau) (v-\bar{u}(\tau)), v-\bar{u}(\tau)\Big\rangle \chi_{A_{ij}}(\omega)\Big] \lambda^*(\mathrm{d}\gamma)  \\
&\indent -\int_{\Gamma}\mathbb{E} \Big[ \Big\langle  \mathbb{S}_\gamma(\tau) \partial_u\sigma_\gamma(\tau) (v-\bar{u}(\tau)),  \nabla\bar{u}(\tau)\Big\rangle \chi_{A_{ij}}(\omega)\Big] \lambda^*(\mathrm{d}\gamma)   \\
&\quad \le0 \quad  \forall(\tau,v)\in([t_i,T)\backslash \cup_{i,j,k\in\mathbb{N}}E_{i,j}^k)\times U.
\end{align*}
Furthermore, by the Assumption (H6) and  Fubini's theorem, it yields 
\begin{align*}
&\mathbb{E}\Big[\int_{\Gamma} \Big\langle\mathbb{S}_\gamma(\tau) \partial_u b_\gamma(\tau) (v-\bar{u}(\tau)),v-\bar{u}(\tau)\Big\rangle \lambda^*(\mathrm{d}\gamma) \chi_{A_{ij}}(\omega) \Big] \\
&\indent  + \mathbb{E}\Big[\int_{\Gamma} \Big\langle  \nabla\mathbb{S}_\gamma(\tau)  \partial_u\sigma_\gamma(\tau) (v-\bar{u}(\tau)), v-\bar{u}(\tau)\Big\rangle \lambda^*(\mathrm{d}\gamma)\chi_{A_{ij}}(\omega)\Big]  \\
&\indent -\mathbb{E} \Big[\int_{\Gamma} \Big\langle  \mathbb{S}_\gamma(\tau) \partial_u\sigma_\gamma(\tau) (v-\bar{u}(\tau)),  \nabla\bar{u}(\tau)\Big\rangle \lambda^*(\mathrm{d}\gamma)\chi_{A_{ij}}(\omega)\Big]   \\
&\quad \le0 \quad  \forall(\tau,v)\in([t_i,T)\backslash \cup_{i,j,k\in\mathbb{N}}E_{i,j}^k)\times U.
\end{align*}
Recall that   $\{A_{ij}\}_{j=1}^{\infty}(\subset\mathcal{F}_{t_i})$ are taken  generating $\mathcal{F}_{t_i}$ and $\mathbb{F}$ is the natural filtration. Then we obtain 
\begin{align*}
&\int_{\Gamma} \Big\langle\mathbb{S}_\gamma(\tau) \partial_u b_\gamma(\tau) (v-\bar{u}(\tau)),v-\bar{u}(\tau)\Big\rangle \lambda^*(\mathrm{d}\gamma)  \\
&\indent  + \int_{\Gamma} \Big\langle  \nabla\mathbb{S}_\gamma(\tau)  \partial_u\sigma_\gamma(\tau) (v-\bar{u}(\tau)), v-\bar{u}(\tau)\Big\rangle \lambda^*(\mathrm{d}\gamma)  \\
&\indent -\int_{\Gamma} \Big\langle  \mathbb{S}_\gamma(\tau) \partial_u\sigma_\gamma(\tau) (v-\bar{u}(\tau)),  \nabla\bar{u}(\tau)\Big\rangle \lambda^*(\mathrm{d}\gamma)  \\
&\quad \le0  \quad a.s.\  \forall(\tau,v)\in([t_i,T)\backslash \cup_{i,j,k\in\mathbb{N}}E_{i,j}^k)\times U.
\end{align*}
The proof is completed.
\end{proof}

\appendix 
\section*{Appendix}

\setcounter{section}{0}
\section{Technical lemmas}

\renewcommand{\thesubsection}{\Alph{subsection}}

The following  Minimax Theorem (cf. \cite{pham09}) describes conditions that guarantee the saddle point property, and it plays an essential role in this paper to derive the weak limit of the uncertainty probability measures. 
\begin{lem}[Minimax Theorem]\label{mimmaxthe}
    Let $\mathcal{M}$  be a convex subset of a normed vector space,  and  $\mathcal{N}$  a weakly compact convex subset of a normed vector space. If  $\varphi$  is a real-valued function on $ \mathcal{M} \times  \mathcal{N} $ with 
    
    $x\to \varphi (x,y)$ is continuous and convex on $\mathcal{M}$ for all $y\in \mathcal{N}$,

    $y\to \varphi (x,y)$ is concave on $\mathcal{N}$ for all $x\in \mathcal{M}$.

    \noindent 
    Then 
\[\sup_{y\in \mathcal{N}}\inf_{x\in \mathcal{M}}\varphi (x,y) =\inf_{x\in \mathcal{M}}\sup_{y\in \mathcal{N}}\varphi (x,y).\]
    \end{lem}

The following Partitions of Unity Theorem (cf.  \cite{rudin06}) is important in obtaining measurability, with which we can approximate target processes by simple adapted ones and then make uniform estimates. 
\begin{lem}[Partitions of Unity Theorem]\label{partunith}
    Denote by $\{\mathcal{V}_{i}\}_{i=1}^n$ any open subsets of a locally compact Hausdorff space $ \mathfrak{M}  $. If a compact set $\mathcal{K} $ satisfies 
    \[\mathcal{K} \subset \mathcal{V}_1 \cup \ldots \cup \mathcal{V}_n, \]
    then there exists $\{\varphi_i\}_{i=1}^n$ satisfying 
    \[\Sigma_{i=1}^n \varphi_i(x)=1,\quad x\in \mathcal{K},\]
    where $\varphi_i$ is continuous on $\mathfrak{M}$, with its compact support lies in $\mathcal{V}_i$, $0\le \varphi_i \le 1$, $i=1,\ldots  n$.
\end{lem}

\renewcommand{\theequation}{B.\arabic{equation}}
{\small 
\section{A comment on (H6) (ii)}
One may wonder whether the Assumption (H6) (ii) used in Theorem~\ref{thm-point-nece-cla} can be obtained with direct estimates utilizing such as the following estimates \eqref{lemconti-sthetetheequ}.
\begin{lem}[{\cite[Lemma 5.1.4]{NualartDavidEulalia}}]\label{lemconti-sthetethe}
Fix an integer $k \ge 1$ and a real number $p > 1$. Then, there
exists a constant $c_{p,k}$ such that, for any random variable $F \in \mathbb{D}^{k,2}$,
\begin{equation}\label{lemconti-sthetetheequ}
\|\mathbb{E}(\mathcal{D}^k F)\|_{H^{\bigotimes k}} \le c_{p,k} \|F\|_p .
\end{equation}
\end{lem} \noindent 
Recall that for adapted process $\mathbb{S}(\cdot)$ and for $0<t<s$, it holds that $\mathcal{D}_s\mathbb{S}(t)=0$. 
Then by Lemma~\ref{lemconti-sthetethe} and \eqref{jhzls}, taking $k=1,p=2$, and $F=(\mathbb{S}_\gamma-\mathbb{S}_{\gamma'})(t)$, together with $H^{\bigotimes 1}=L^2(0,T)$, we get 
\begin{align*} 
    \int_{0}^{T}\left|\mathbb{E}\left[\mathcal{D}_s\mathbb{S}_\gamma(t)-\mathcal{D}_s\mathbb{S}_{\gamma'}(t)\right]\right|^2\mathrm{d}s & = \int_{0}^{t}\left|\mathbb{E}\left[\mathcal{D}_s\mathbb{S}_\gamma(t)-\mathcal{D}_s\mathbb{S}_{\gamma'}(t)\right]\right|^2\mathrm{d}s 
\lesssim \mathbb{E} |\mathbb{S}_\gamma(t)-\mathbb{S}_{\gamma'}(t)|^2  
\end{align*}
and then 
$\int_{0}^{T} \int_{0}^{t}\left|\mathbb{E}\left[\mathcal{D}_s\mathbb{S}_\gamma(t)-\mathcal{D}_s\mathbb{S}_{\gamma'}(t)\right]\right|^2\mathrm{d}s\mathrm{d}t \lesssim \int_{0}^{T} \mathbb{E} |\mathbb{S}_\gamma(t)-\mathbb{S}_{\gamma'}(t) |^2 \mathrm{d}t$, 
and at last 
\begin{equation}\label{lethetheprocontisupreop}
\begin{aligned}
&\lim_{\ell \to0} \sup_{ \mathsf{d} (\gamma,\gamma')\le\ell } \int_{0}^{T} \int_{0}^{t}\left|\mathbb{E}\left[\mathcal{D}_s\mathbb{S}_\gamma(t)-\mathcal{D}_s\mathbb{S}_{\gamma'}(t)\right]\right|^2\mathrm{d}s\mathrm{d}t    \\
& \indent =\lim_{\ell \to0} \sup_{ \mathsf{d} (\gamma,\gamma')\le\ell } \int_{0}^{T} \int_{0}^{T}\left|\mathbb{E}\left[\mathcal{D}_s\mathbb{S}_\gamma(t)-\mathcal{D}_s\mathbb{S}_{\gamma'}(t)\right]\right|^2\mathrm{d}s\mathrm{d}t =0.
\end{aligned}
\end{equation}
However,~\eqref{lethetheprocontisupreop} is merely about the weak continuity of $\mathcal{D}_s\mathbb{S}_\gamma(t)$ in $\gamma$  uniformly with respect to $(s,t)\in[0,T]\times[0,T]$ and is not sufficient to carry out the partitions of unity method to obtain the measurability of $\mathcal{D}_\cdot \mathbb{S}_\cdot (\cdot )$.
Actually, to perform the method, one needs the following  stronger uniform continuity:
\begin{equation}\label{lethetheprocontisupreop-nee}
\begin{aligned}
&\lim_{\ell \to0} \sup_{ \mathsf{d} (\gamma,\gamma')\le\ell }  \int_{0}^{T}\int_{0}^{T} \mathbb{E}\big( |\mathcal{D}_s\mathbb{S}_\gamma(t)-\mathcal{D}_s\mathbb{S}_{\gamma'}(t) |^2\big)  \mathrm{d}s\mathrm{d}t =0.
\end{aligned}
\end{equation}
To obtain~\eqref{lethetheprocontisupreop-nee}, one needs an estimate in the form of~\eqref{lemconti-sthetetheequ} but with the expectation being out of the Hilbert norm, which is not true in general.
}

{\small \section*{Funding and  Competing interests}
The authors have no competing interests to declare that are relevant to the content of this article. The authors have no relevant financial or non-financial interests to disclose. Research supported by National Key R\&D Program of China (No. 2022YFA1006300) and the NSFC (No. 12271030).  }

{\small \section*{Acknowledgments}
The authors are grateful to the editors and referees for helpful comments that have improved the paper.}

\end{document}